\documentclass[final]{siamltex}
\usepackage{}
\usepackage[dvipdfm]{graphicx}
\usepackage{multirow,bigstrut}
\usepackage{amsmath,amsfonts,amssymb}
\usepackage{mathrsfs}
\usepackage{color}

\newtheorem{niao}{\hspace{1mm}Theorem}[section]
\newtheorem{lem}{\hspace{1mm}Lemma}[section]

\newcommand{\beq}{\begin{equation}}
\newcommand{\eeq}{\end{equation}}
\newcommand{\bey}{\begin{eqnarray}}
\newcommand{\eey}{\end{eqnarray}}

\newcommand{\txn}{\textnormal}

\title{High order fast algorithm for the Caputo fractional derivative\thanks{ This work is supported
by the National Natural Science Foundation of China (grants \#
11501554, 91630205), and  the Fundamental Research Funds for the Central Universities (project \# 106112017CDJXY100006).}}

\author{Kun Wang\thanks
{College of Mathematics and Statistics, Chongqing University, Chongqing, 401331, China ({\tt kunwang@cqu.edu.cn}).}
\and
Jizu Huang\thanks{LSEC, Institute of Computational Mathematics and
    Scientific/Engineering Computing, Academy of Mathematics and Systems
    Science, Chinese Academy of Sciences, Beijing(100190), China
    ({\tt huangjz@lsec.cc.ac.cn}).} Corresponding author:
huangjz@lsec.cc.ac.cn.
  }

\begin{document}
\maketitle
\begin{abstract}
{In the paper, we present a high order fast algorithm with almost optimum memory for the Caputo fractional derivative,
which can be expressed as a convolution of $u'(t)$ with the kernel $(t_n-t)^{-\alpha}$.
In the fast algorithm, the interval $[0,t_{n-1}]$ is split into  nonuniform subintervals. 
The number of the subintervals is in the order of $\log n$ at the $n$-th time step.
The fractional kernel function is approximated by a polynomial function of $K$-th degree with a uniform absolute error on each subinterval.
We save $K+1$ integrals on each subinterval, which can be written as a convolution of $u'(t)$ with a polynomial base function.
As compared with the direct method, the proposed fast algorithm reduces the storage requirement  and computational cost from $O(n)$ to $O((K+1)\log n)$
at the $n$-th time step.
We prove that the convergence rate of the fast algorithm is the same as the direct method even a high order direct method
is considered. The convergence rate and efficiency of the fast algorithm are illustrated via several numerical examples.}
\end{abstract}


\section{Introduction}
\label{sec1}
In recent years, the fractional differential equation becomes popular
since they can faithfully capture the dynamics of physical process in 
many scientific phenomena, such as the dynamics of biology, ecology, and control system 
\cite{jiang,kilbas,langlands,langlands10,li09,lin11,metzler,oldham,podlubny99,ren}. 
There are mainly two kinds of definitions of the fractional time derivative in the literatures:  
the Riemann--Liouville fractional derivative \cite{chen,cui} and the
Caputo fractional derivative \cite{jiang,ren,sun06,yang,zhang}. 
In fractional partial differential equations (PDEs), the time fractional derivatives are commonly defined using the
Caputo fractional derivatives since the Riemann--Liouville approach needs initial conditions containing the limit values 
of Riemann--Liouville fractional derivative at the origin of time $t=0$, 
whose physical meanings are not very clear.
In the paper, 
we focus on the high order fast method of the PDEs including the Caputo fractional derivative which is defined by
\begin{align}\label{1.1}
^C_0{\cal D}^\alpha_tu(t)=\frac{1}{\Gamma(1-\alpha)}\int\limits_0^t\frac{u'(x,\tau)}{(t-\tau)^\alpha}\txn d\tau,~~~~~0<\alpha<1,
\end{align}
where $\Gamma(\cdot)$ is the gamma function and $t$ is in $[0,T]$.

One of the popular schemes of discretizing the Caputo fractional derivative is usually called $L1$ formula \cite{gao12, lin07, sun06},
which applies the piecewise linear interpolation of $u(x,t)$ with respect to $t$ in the integrand on each subinterval.
For $0<\alpha<1$, the scheme enjoys a $2-\alpha$  order of convergence rate. 
Some other methods with a $2-\alpha$  order of convergence rate are also studied, such as the  Crank--Nicolson-Type discretization \cite{zhang} 
and  the matrix transfer technique  \cite{yang}. 
By applying the fractional linear multistep methods  in discretizing the Caputo fractional derivative, 
an exactly second order scheme with unconditional stability is constructed in \cite{zeng}. 
By using the piecewise quadratic interpolation of $u(t)$ in the integrand  for the Caputo fractional derivative,  
Gao and Sun \cite{gao14} propose a new discrete formula (called $L1-2$ formula)  which achieves $3-\alpha$ order accuracy. 
Recently, based on the block-by-block approach, Cao et al. improve the discretization in time and  a  scheme with
order $3+\alpha$ is successfully constructed in \cite{cao}. On the other hand, a scheme with spectral accuracy is also investigated  in \cite{li09}.
These direct methods require the storage of all previous solutions, which leads to $O(n)$ storage and $O(n)$ flops at the $n$-th time step. 
Therefore, an efficient and reliable fast method is needed for long time large scale simulation of fractional PDEs. 
 
 In order to save memory and computational cost, some fast methods are developed.  
 In \cite{lubich2002}, a fast convolution method for the Caputo fractional derivative is proposed, in which
 the kernel function is first expressed by it inverse Laplace transform. The idea is then extended to calculate the 
 Caputo fractional  derivative in \cite{lopez2008, schadle2006, zeng2017}.
The storage requirement  and the computational
 cost of those fast methods both are $O(\log n)$ at the $n$-th time step, which are less than that of the direct methods.
 In \cite{ren}, the Laplace transform method is used to transforme the fractional differential equation into an approximation local problem.
 In \cite{lijr2010}, the Gauss--Legendre quadrature is applied to construct a fast algorithm based on the formula
 $t^{\alpha-1}=\frac{1}{\Gamma(\alpha)\Gamma(1-\alpha)}\int_0^\infty e^{-\xi t}\xi^{-\alpha}\txn d \xi$.
 The fast method is improved by Jiang et al. \cite{jiang} by using the Gauss--Jacobi and Gauss--Legendre quadratures together,
 which only requires  the storage  and the computational
 cost in the order of $O(\log n)$ at the $n$-th time step.
The fast scheme is proved to be unconditionally stable and has a convergence order of $2-\alpha$ \cite{jiang}.
In \cite{mclean2012}, McLean proposes a fast method to approximate the fractional integral by replacing the fractional kernel with a degenerate kernel.
In \cite{Baffet}, a kernel compression method is presented to discretize the fractional integral operator, which is based on multipole
approximation to the Laplace transform of the fractional kernel.

In this paper, we aim to present a high order fast algorithm with almost optimum memory for the Caputo
fractional derivative, which  has the same order of convergence rate as that of
a given direct method. At each time step, the fractional derivative is  decomposed into the local part and the history part.
The local part, which is an integral on interval $[t_{n-1},t_n]$, is calculated by a direct method. In order to evaluate the history part by 
a high efficient approach with low cost,
we split the interval $[0,t_{n-1}]$ into nonuniform subintervals at the $n$-th time step. The total number of the subintervals
is in the order of $\log n$. We save $K+1$ integrals on each subinterval and evaluate the history part with those integrals.
To reuse the storages in the previous time step, we approximate the fractional kernel function by a polynomial function
with a uniform absolute error on the subintervals. As compared with a direct method based on a given polynomial interpolation 
of $u(t)$, the new proposed fast method is proved to enjoy the same convergence order
by controlling the absolute error of the approximate polynomial function,
but only requires computational storage and flops in the order of $\log n$ at the $n$-th time step.

The remainder of the paper is organized as follows. In Section 2, we describe the high order fast algorithm for 
the evolution of the Caputo fractional derivative and provide error analysis of our method. 
In Sections 3 and 4, we apply the proposed high order fast algorithm to solve the linear and nonlinear fractional diffusion PDEs.
The stability and numerical error analysis for the new algorithm and some existing methods are carefully studied.
The numerical results demonstrate that our high order algorithm has the same convergence order as
the corresponding direct method. Finally, some brief conclusions are given in Section 5.

\section{High order fast algorithm with almost optimum memory of the Caputo fractional derivative} \label{intro-sec}
In this section,
we consider the high order fast algorithm with almost optimum memory for the evolution of the Caputo fractional derivative,
which is defined as in (1.1).
Suppose that the time interval $[0,T]$ is covered by a set of grid points $\Omega_t:=\{t_n,\,n=0,1,\cdots,N_T\}$, with $t_0=0$, $t_{N_T}=T$,
$t_{n+\frac 1 2}=\frac{t_n+t_{n+1}}{2}$, and 
$\Delta t_n=t_n-t_{n-1}$. For simplify, 
we only consider a uniform distribution of the grid points which means $\Delta t_n=h$ for all $n$.
We will simply denote $u(t_n)$ by $u^n$.
Let us denote the piecewise linear interpolation function of $u(t)$
as $\Pi_{1,h}u(t)$ for any $j\geq 1$, i.e.,
$$
\Pi_{1,h}u(t)=u^{j-1}\frac{t_j-t}{h}+u^j\frac{t-t_{j-1}}{h},~~\txn{for}~t\in[t_{j-1},t_j].
$$
Suppose $\Pi_{2,h}u(t)$ be a piecewise quadratic interpolation function of $u(t)$ for $j\geq2$, which is given as
\begin{equation}
\begin{aligned}
\Pi_{2,h}u(t)=&u^{j-2}\frac{(t-t_{j-1})(t-t_j)}{2h^2}+u^{j-1}\frac{(t-t_{j-2})(t_j-t)}{h^2}
\\&+u^{j}\frac{(t-t_{j-1})(t-t_{j-2})}{2h^2},~~\txn{for}~t\in[t_{j-1},t_j].
\end{aligned}
\end{equation}
It follows from the interpolation theory that $\Pi_{1,h}u(t)$ and $\Pi_{2,h}u(t)$ have a second-order accuracy
and third-order accuracy in time for smooth $u(t)$, respectively. Let us denote $\big(\Pi_{1,h}u(t)\big)'$ and 
$\big(\Pi_{2,h}u(t)\big)'$ as follows
$$
\big(\Pi_{1,h}u(t)\big)'=\delta_t u^{j-\frac12},~~\txn{for}~t\in[t_{j-1},t_j],
$$
and
$$
\big(\Pi_{2,h}u(t)\big)'=\delta_t u^{j-\frac12}+\delta_t^2u^{j-1}(t-t_{j-\frac12}),~~\txn{for}~t\in[t_{j-1},t_j],
$$
respectively. Here $\delta_t u^{j-\frac12}=\frac{u^j-u^{j-1}}{h}$
and $\delta_t^2u^j=\frac{\delta_t u^{j+\frac12}-\delta_t u^{j-\frac12}}{h}$.

For simplicity, we denote $\Pi_{2,h}u(t)=\Pi_{1,h}u(t)$ for $t\in[t_0,t_1]$.
For $0<\alpha<1$, the most popular scheme for calculating of the Caputo fractional derivative
is called the $L1$ formula \cite{gao12,lin07,sun06}, whose accuracy is $2-\alpha$ order in time.
In the $L1$ formula, $u(t)$ is replaced by the piecewise linear function $\Pi_{1,h}u(t)$ (as shown in Fig.~\ref{fig:1-1}-(a)).
Another popular high order scheme ($L1-2$ formula) achieves  $3-\alpha$ order accuracy \cite{gao14}, 
in which $u(t)$ is approximated by the linear interpolation function $\Pi_{1,h}u(t)$ at interval $[t_0,t_1]$ and
the quadratic interpolation function $\Pi_{2,h}u(t)$ at interval $[t_j,t_{j+1}]$ for $j\geq 1$.
It is well known that the $L1$ formula and $L1-2$ formula require the storage of all previous function 
values of $u^0, u^1,\cdots,u^n$ and 
$O(n)$ flops computational cost at the $(n+1)$-th time step.
For a long time simulation, the direct schemes require very large storage of  memory and high computational cost.

Now let us take the Caputo fractional derivative  (\ref{1.1}) as a convolution integral, 
in which $\frac{1}{(t-\tau)^\alpha}$ can be viewed as a kernel function (weight function). 
For $0<\alpha<1$, the kernel function increases as $\tau$ goes from $0$ to $t$. 
To save memory and computational cost, a natural idea is to cut the integral by a given integer $\bar{S}$ at the $n$-th time step, which means
\begin{equation}
\int\limits_0^{t_n}\frac{ u'(\tau)}{(t_n-\tau)^{\alpha}}\txn d \tau\approx \int\limits^{t_n}_{t_{j_0}}\frac{ u'(\tau)}{(t_n-\tau)^{\alpha}}\txn d \tau
\approx \sum_{j=j_0+1}^{n}\delta_tu^{j-\frac{1}{2}}\int\limits^{t_j}_{t_{j-1}}\frac{1}{(t_n-\tau)^{\alpha}}\txn d \tau,
\end{equation}
where  $u(\tau)$ is approximated by $\Pi_{1,h}u(\tau)$ and $j_0=\max\{0,n-\bar S\}$.
In the following of the paper, we denote this approximation as the cut off approach (as shown in Fig.~\ref{fig:1-1}-(b)). 
It is important to noting that the cut off approach 
only need limited memory according to the given integer $\bar{S}$ for any large $N_T$.
However, the numerical simulation shows that the accuracy of the cut off approach is unacceptable. 

\begin{figure}[htb!]
\begin{center}
{\includegraphics[width=0.47\textwidth]{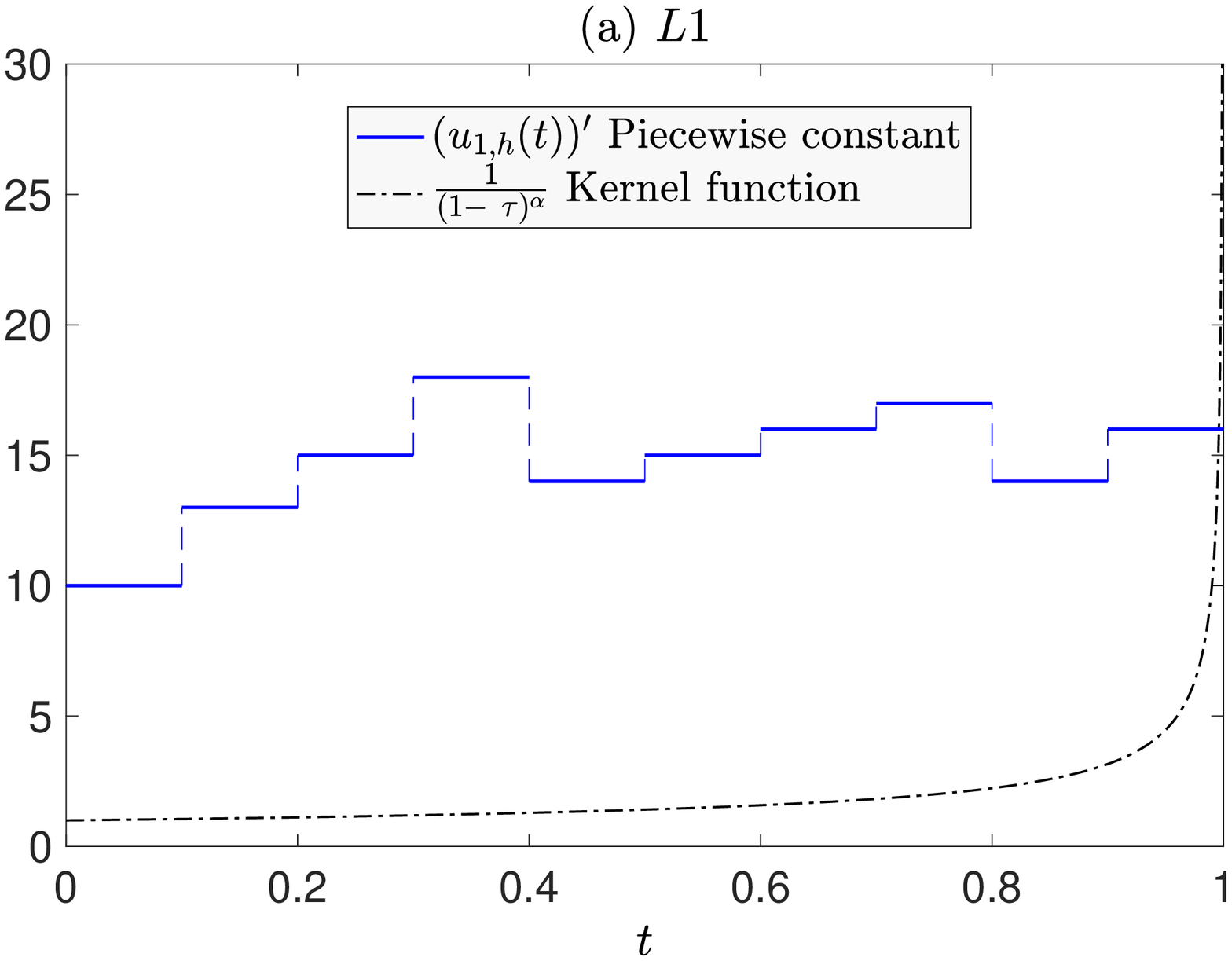}}
{\includegraphics[width=0.47\textwidth]{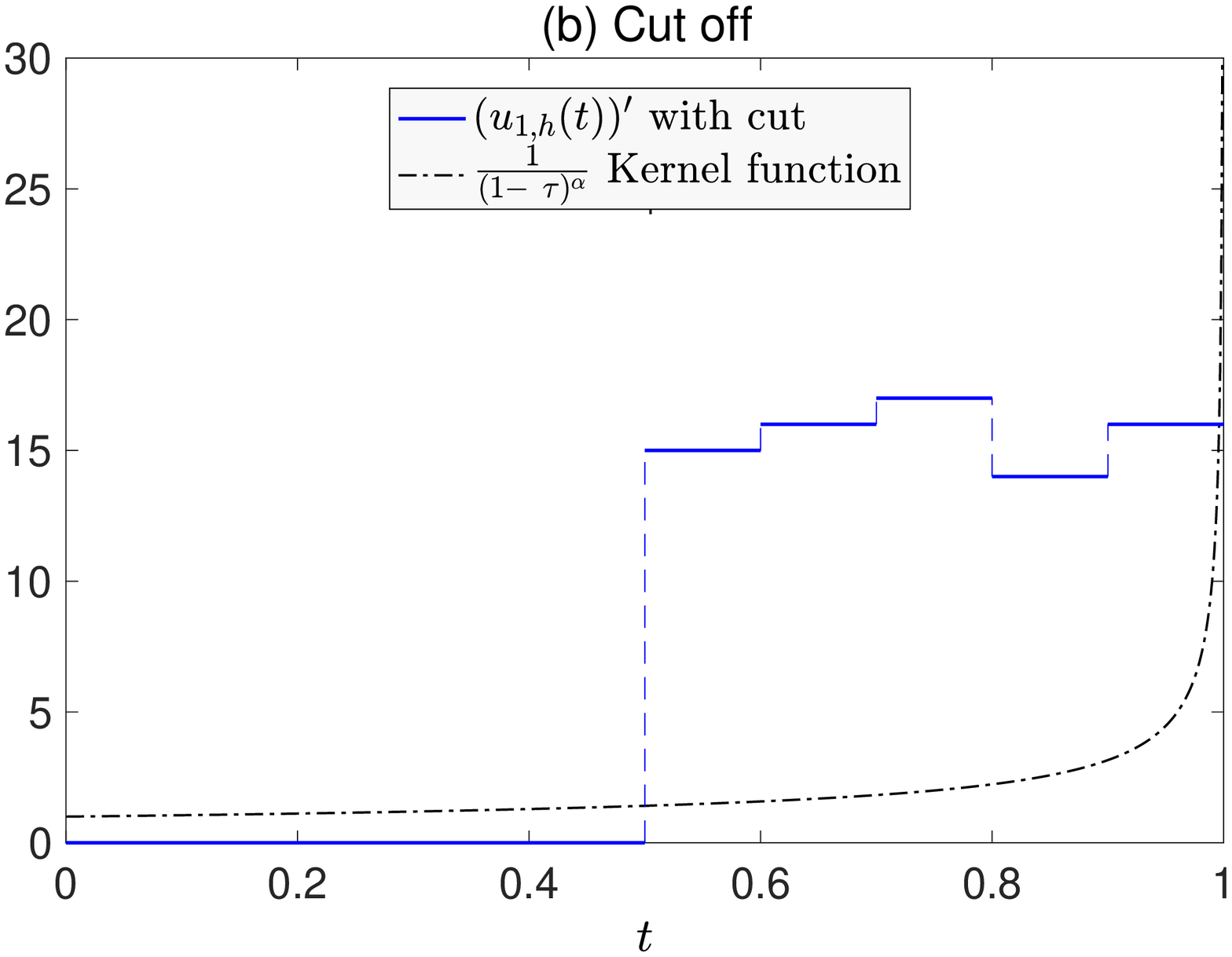}}
{\includegraphics[width=0.47\textwidth]{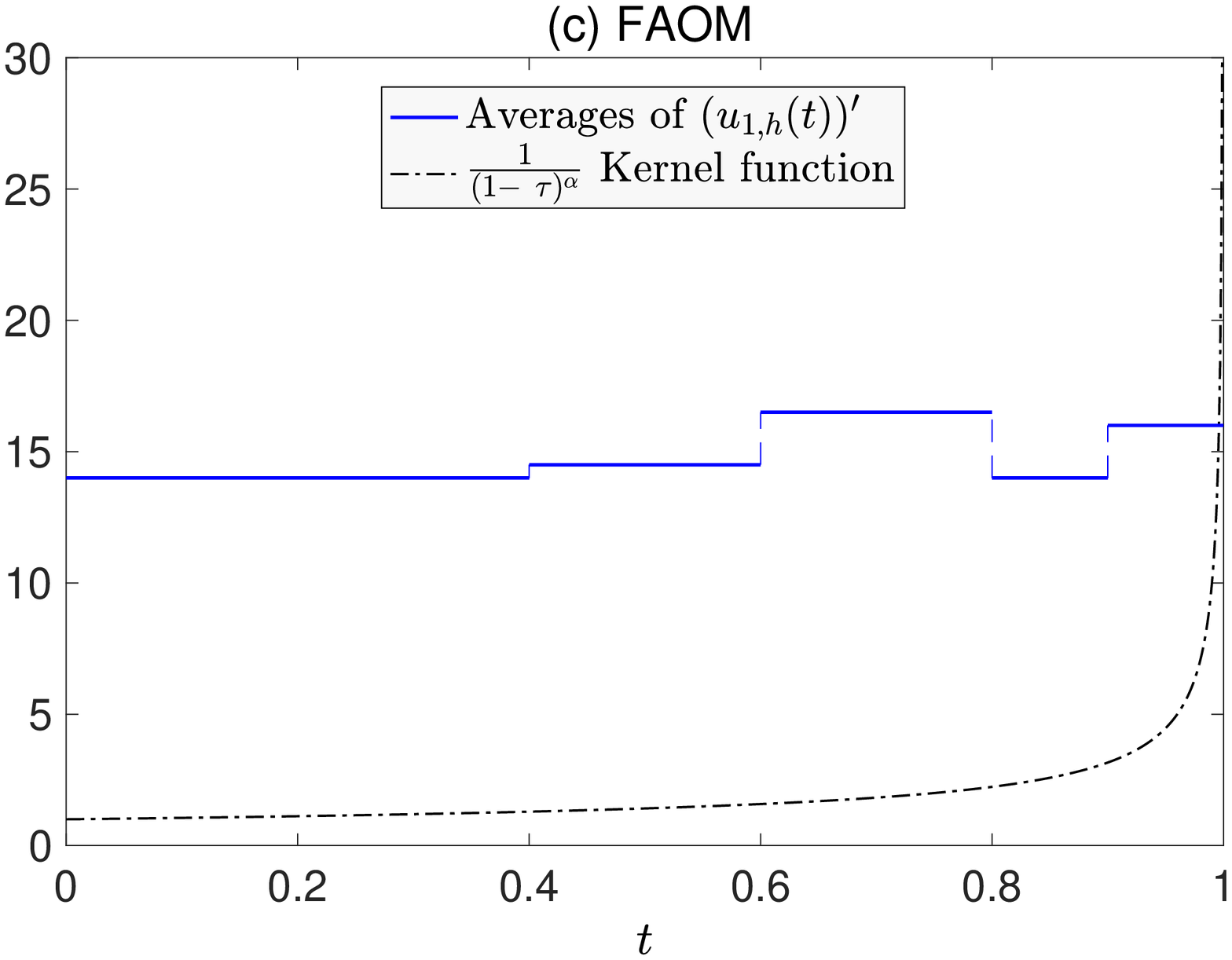}}
{\includegraphics[width=0.47\textwidth]{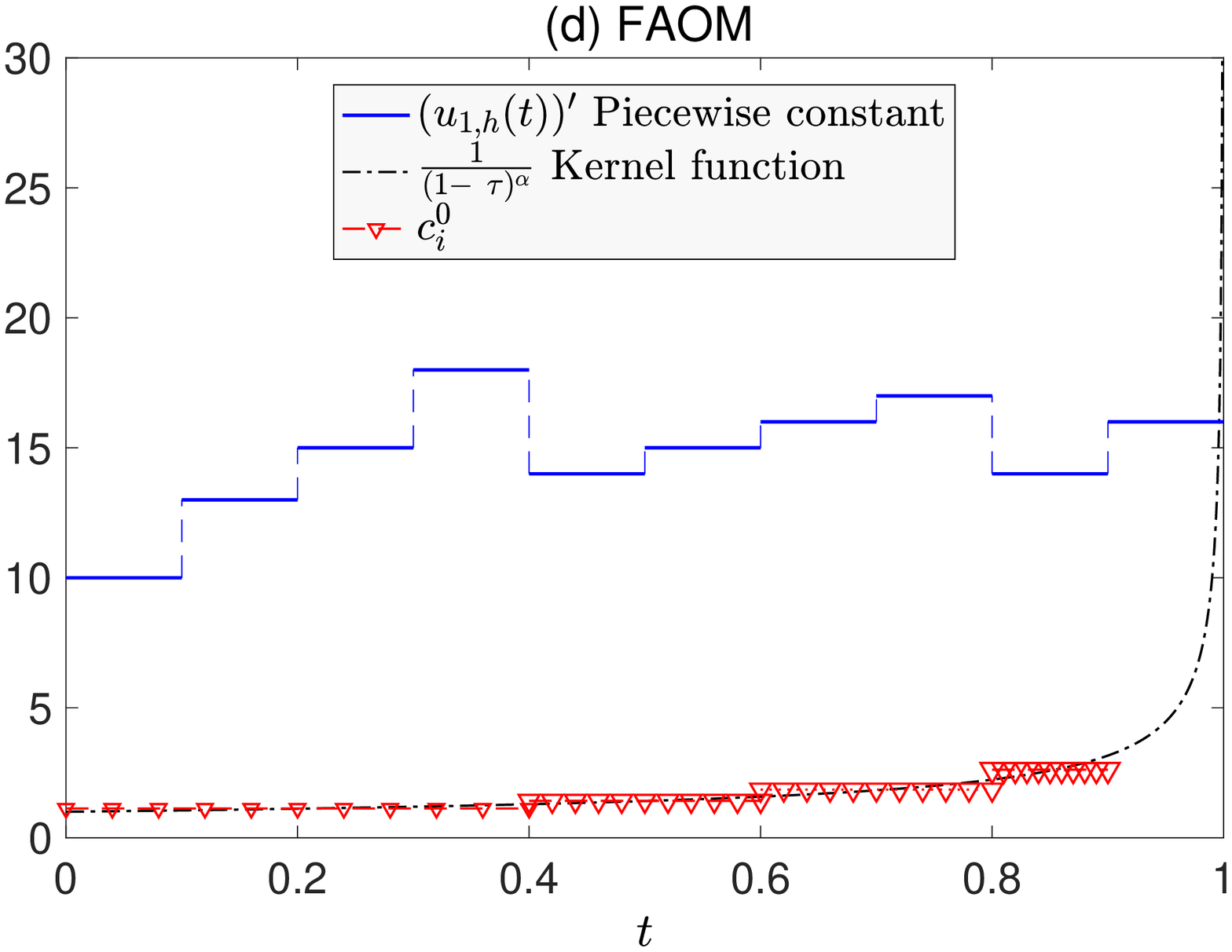}}
\end{center}
\caption{Example for the $L1$ formula (a), the cut off approximation (b), and the FAOM algorithm (c-d).
$T=1$, $h=0.1$. In the cut off approximation, $\bar{S}=5$. $j_1=0$, $j_2=4$, $j_3=6$, $j_4=8$, $j_5=9$
in the FAOM algorithm.
  }\label{fig:1-1}
\end{figure}

We next 
present our  fast evaluation method based on  the understanding of
the $L1$ formula and the cut off approach.
For convenience,  we first introduce some useful definitions in here.
For any given vector ${\mathbf V}$, let us define a backward operator ${\cal B}$ by ${\mathbf U}={\cal B}(v,{\mathbf V})$ with
${U}_1=v$ and ${ U}_i={ V}_{i-1}$ for $i=2,3,\cdots,m+1$. Here
$m$ is the length of ${\mathbf V}$. Let ${\cal F}$ be a forward operator defined by
${\mathbf U}={\cal F}(M,j,{\mathbf V})$, where ${ U}_i={ V}_i$ for $i<j$ and ${ U}_i={ V}_{i+M-1}$
for $i\geq j$ with $M\geq2$. Let ${\cal F}_m$ be a modified forward operator defined by
${\mathbf U}={\cal F}_m(M,j,v,{\mathbf V})$, where ${ U}_i={ V}_i$ for $i<j$, ${ U}_i=v$
for $i=j$, and ${ U}_i={ V}_{i+M}$ for $i> j$.

In the cut off approach, the solutions at the previous several time steps  are saved since those solutions are important and correspond to
large weight functions. However, the numerical simulation suggests that we should take the other solutions  into account,
even those solutions correspond to small weight functions. To balance the storage of memory and the accuracy of solution, 
we save the averages of  $(u_{J,h}(\tau))'$
in the nonuniform subintervals $[t_{j_i},t_{j_{i+1}}]$ in the fast evolution algorithm (as shown in Fig.~\ref{fig:1-1}-(c)).
It is clear that we hope the averages of $(u_{J,h}(\tau))'$ at the previous time steps 
can be reused in the current time step and the following time steps.
Furthermore, the length of the subinterval $[t_{j_i},t_{j_{i+1}}]$
should  decrease as $i$ increases, since  the kernel function  is an increasing function.
By using an interpolation $\Pi_{J,h}u(\tau)$ for $u(t)$, the integral in equation (\ref{1.1}) 
at time $t=t_n$ can be  approximated as follows
\begin{equation}\label{eq:al-1}
\begin{aligned}
\int\limits_0^{t_n}&\frac{ u'(\tau)}{(t_n-\tau)^{\alpha}}\txn d \tau=\int\limits_{t_{n-1}}^{t_n}\frac{ u'(\tau)}{(t_n-\tau)^{\alpha}}\txn d \tau
+\int\limits_0^{t_{n-1}}\frac{ u'(\tau)}{(t_n-\tau)^{\alpha}}\txn d \tau:={\cal I}_l(t_n)+{\cal I}_h(t_n)\\
&\approx 
\int\limits_{t_{n-1}}^{t_n}\frac{\big(\Pi_{J,h} u(\tau)\big)'}{(t_n-\tau)^{\alpha}}\txn d \tau
+\int\limits_0^{t_{n-1}}\frac{\big(\Pi_{J,h} u(\tau)\big)'}{(t_n-\tau)^{\alpha}}\txn d \tau
\\&\approx\int\limits_{t_{n-1}}^{t_n}\frac{\big(\Pi_{J,h} u(\tau)\big)'}{(t_n-\tau)^{\alpha}}\txn d \tau+\sum\limits_i \int\limits_{t_{j_i}}^{t_{j_{i+1}}}
\frac{\big(\Pi_{J,h} u(\tau)\big)'}{t_{j_{i+1}}-t_{j_i}}\txn 
d \tau\int\limits_{t_{j_i}}^{t_{j_{i+1}}}\frac{1}{(t_n-\tau)^{\alpha}}\txn d \tau.
\end{aligned}
\end{equation}
Here the integral is decomposed to the local part ${\cal I}_l(t_n)$ and the history part ${\cal I}_h(t_n)$.
The balance between the  storage of memory and the accuracy of solution can be done by choosing suitable subintervals $[t_{j_i},t_{j_{i+1}}]$.
It is worth to pointing out that  this approximation is the same with the $L1$ formula by setting $t_{j_{i+1}}-t_{j_i}=h$ and $J=1$.
We now propose a fast evolution approach (Algorithm 1) to reach an almost optimum memory by  constructing a special sequence of subintervals. 
The approach is named as the fast algorithm with almost optimum memory
(FAOM) of the Caputo fractional derivative.


\noindent
\noindent\underline{{~~~~~~~~~~~~~~~~~~~~~~~~~~~~~~~~~~~~~~~~~~~~~~~~~~~~~~~~~~~~~~~~~~~~~~~~~~~~~~~~~~~~~~~~~~~~~~~~~~~~~~~~~~~~~~~}}

\noindent\textbf{Algorithm 1}: FAOM of the Caputo fractional derivative. 
\vspace{-0.3cm}

\noindent\underline{{~~~~~~~~~~~~~~~~~~~~~~~~~~~~~~~~~~~~~~~~~~~~~~~~~~~~~~~~~~~~~~~~~~~~~~~~~~~~~~~~~~~~~~~~~~~~~~~~~~~~~~~~~~~~~~~}}

\vspace{0.15cm}
\noindent\textbf{Initialization:} Let ${\mathbf U}^n=[U_1^n,U_2^n,\cdots,U^n_{{\cal M}_n}]$ 
be a vector, whose elements are the averages of $ (u_{J,h}(\tau))'$ on given subintervals,
  and ${\mathbf I}^n=[I^n_1,I^n_2,\cdots,I^n_{{\cal M}_n}]$ be a vector, whose elements are the starts of subintervals. Set $\mathbf U^0=0$
  and $\mathbf I^0=0$. Pre-chosen an integer  ${\cal N}_{\tau}\geq 2$ to control the storage of memory. 
  
\noindent\textbf{Start time loop:} $t_n=nh$ with $n=2,\,3,\ldots,\,N_T$.
\begin{description}
  \item[~~Step 1] (Updating the storage):  Update the temporary storage vector by  $\tilde{\mathbf U}^n=
  {\cal B}\big(\tilde u_{a},\mathbf U^{n-1}\big)$
  and  vector ${\tilde {\mathbf I}}^n={\cal B}(t_{n-2},{\mathbf I}^{n-1})$,
  where $\tilde u_{a}=\frac{1}{h}{\int\limits_{t_{n-2}}^{t_{n-1}}\big(\Pi_{J,h} u(\tau)\big)'\txn d\tau}$.
    \item[~~Step 2] (Optimizing of the storage): Obtain the storage vector $\mathbf U^n$ and vector 
    $\mathbf I^n$ as follows.   
  \end{description}
\begin{itemize}
  \item If there exists $i_0$ such that $\tilde I^n_{i_0}-\tilde I^n_{i_0+1}=\cdots=
   \tilde I^n_{i_0+2{\cal N}_{\tau}-2}-\tilde I^n_{i_0+2{\cal N}_{\tau}-1}$, let $\mathbf I^n={\cal F}({\cal N}_{\tau},i_0+{\cal N}_{\tau},
   \tilde{\mathbf I}^n)$ and $\mathbf U^n={\cal F}_m({\cal N}_{\tau},i_0+{\cal N}_{\tau}, v,
   \tilde{\mathbf U}^n)$, where $\tilde I^n_0=I^n_0=t_{n-1}$ and
   $v=\frac{1}{{\cal N}_\tau}\sum\limits_{i=i_0+{\cal N}_\tau}^{i_0+2{\cal N}_\tau-1}\tilde {U}^n_i$. 
 Set $\tilde{\mathbf I}^n={\mathbf I}^n$, $\tilde{\mathbf U}^n={\mathbf U}^n$ and redo optimization
   until there does not exist $i_0$ satisfying $I^n_{i_0}-I^n_{i_0+1}=\cdots=
   I^n_{i_0+2{\cal N}_{\tau}-2}-I^n_{i_0+2{\cal N}_{\tau}-1}$.
  \end{itemize}
  \begin{description}
  \item[~~Step 3] (Calculating the Caputo fractional derivative):  
   Approximate the history part ${\cal I}_h(t_n)$ as follows
\begin{equation}
\begin{aligned}
{\cal I}_h(t_n)=\int\limits_0^{t_{n-1}}\frac{ u'(\tau)}{(t_n-\tau)^{\alpha}}\txn d \tau\approx\sum\limits_{i=0}^{{\cal M}_n-1}
U^n_{i+1}\int\limits_{I^n_{i+1}}^{I^n_{i}}\frac{1}{(t_n-\tau)^{\alpha}}\txn d \tau,
\end{aligned}
\end{equation}
where ${\cal M}_n$ is the length of the vector $\mathbf U^n$. The numerical Caputo fractional derivative is finally calculated according to (\ref{eq:al-1}). 
  \end{description}

\noindent\textbf{End of time loop}.

\vspace{-0.2cm}
\noindent\underline{{~~~~~~~~~~~~~~~~~~~~~~~~~~~~~~~~~~~~~~~~~~~~~~~~~~~~~~~~~~~~~~~~~~~~~~~~~~~~~~~~~~~~~~~~~~~~~~~~~~~~~~~~~~~~~~~}}
\vspace{0.2cm}

\begin{figure}[htb!]
\begin{center}
{\includegraphics[width=0.92\textwidth]{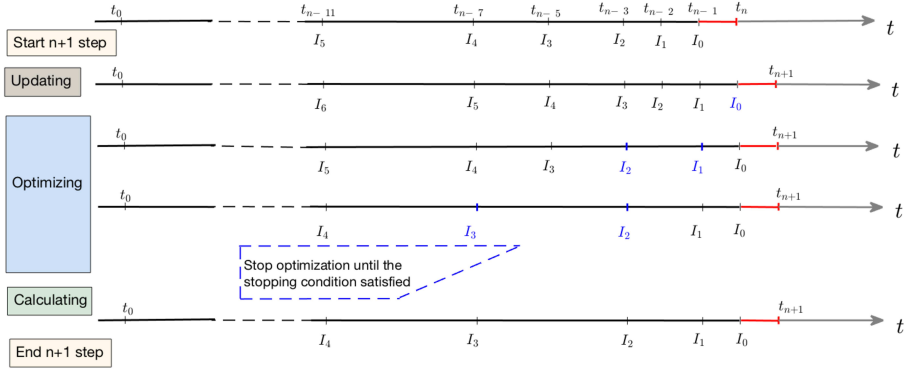}}
\end{center}
\caption{A simple example to explain how we update the  vector $\mathbf I^{n+1}$ at the $(n+1)$-th time step. The optimization of  
 the storage vector $\mathbf U^{n+1}$ is similar. In this example, we set ${\cal N}_{\tau}=2$.
 The part with red color in the time axle is denoted as the  local part ${\cal I}_l(t_n)$ and the part with black color is denoted
 as the history part ${\cal I}_h(t_n)$.}\label{fig:1-2}
\end{figure}

In the following of the paper, we simplify $I_i^n$ as $I_i$ in the absence of ambiguity. 
In the FAOM method, the storage requirement and the overall computational cost both are dependent on the number of
the nonuniform subintervals, which is equal to the length of the vector $\mathbf U^n$.
At the $n$-th time step, the nonuniform subintervals  obtained by the FAOM method satisfy the following properties.
\begin{enumerate}
\item The union of  all subintervals ($\cup_{i=1}^{{\cal M}_n}[I_{i},I_{i-1}]$)
equals $[0,t_{n-1}]$.
\item The length of the subinterval $[I_{i},I_{i-1}]$ is equal to that of the subinterval $[I_{i-1},I_{i-2}]$ or  ${\cal N}_{\tau}$ times of it.
\item The length of the subinterval $[I_{i},I_{i-1}]$ is $h\big({\cal N}_{\tau}\big)^{K_i} $ 
with $K_i\in \mathbb {Z}^+$. Here $\{K_i\}_{i=1}^{{\cal M}_n}$ is a descending sequence.
\item For any $K_i< K_{{\cal M}_n}$, there are at least ${\cal N}_{\tau}-1$ and at most $2{\cal N}_{\tau}-2$
subintervals, whose length are
$h\big({\cal N}_{\tau}\big)^{K_i} $.
\item{ Most of the subintervals at the previous time step are unchanged in the current time step.  }
\end{enumerate}
To further describe the approach  clearly, we take a special case  as an example and show how the 
subintervals change from the $n$-th time step to the $(n+1)$-th time step in Fig.~\ref{fig:1-2}.
While there are $2{\cal N}_\tau-1$ subintervals with the same length, the FAOM algorithm combines 
 ${\cal N}_\tau$ of them to a large subinterval during the optimizing step. 
The following lemmas show the relationship between the length of the vector $\mathbf U^n$ and $n$.

\begin{lem}\label{lem:1}
Let $I_{i-1}-I_{i}=h\big({\cal N}_{\tau}\big)^{K_i}$ for $i=1,\,2\,\cdots,\,{\cal M}_n$.
The following inequalities hold 
\begin{equation}\label{eq:tem:1}
\begin{aligned}
&\frac{i}{2{\cal N}_\tau-2}-1\leq K_{i}\leq \frac{i-1}{{\cal N}_{\tau}-1},\\
&\frac{\tau-I_{i}}{t_n-I_{i}}\leq\frac 12,\quad   \forall \tau\in[I_{i},I_{i-1}].
\end{aligned}
\end{equation}
\end{lem}
\begin{proof}
From the second and fourth properties listed above, we obtain that $
\big({\cal N}_{\tau}\big)^{i/(2{\cal N}_{\tau}-2)-1}\leq
\frac{I  _{i-1}-I  _{i}}{h}=\big({\cal N}_{\tau}\big)^{K_i}\leq \big({\cal N}_{\tau}\big)^{(i-1)/({{\cal N}_{\tau}}-1)}$ holds for $i\geq1$,
which shows the first inequality in (\ref{eq:tem:1}) holds.
From the properties 1, 2, and 4, we get
\begin{equation}
\frac{t_{n}-I  _{i}}{h}
=1+\sum\limits_{k=1}^{i}\big({\cal N}_{\tau}\big)^{K_k}\geq1+\big({\cal N}_{\tau}\big)^{K_i}+({\cal N}_{\tau}-1)\sum\limits_{k=0}^{K_i-1}\big({\cal N}_{\tau}\big)^{k}=2\big({\cal N}_{\tau}\big)^{K_i},
\end{equation}
which implies $\frac{\tau-I  _{i}}{t_n-I  _{i}}\leq\frac 12$ holds for any $\tau\in[I  _{i},I  _{i-1}]$.
\end{proof}
\begin{lem}\label{lem:2.2}
At the $n$-th time step, the length of the vector $\mathbf U^n  $ satisfies 
\begin{equation}
({\cal N}_{\tau}-1)\big(\log_{{\cal N}_{\tau}}{n}-1\big)\leq {\cal M}_n\leq2({\cal N}_{\tau}-1)\log_{{\cal N}_{\tau}}\Big(\frac{n+1}{2}\Big).
\end{equation}
\end{lem}
\begin{proof}
At the $n$-th time step, it is clear $I  _1=t_{n-2}$ and $I  _{{\cal M}_n}=0$.
Let $[a_{i+1},a_{i}]$ be a sequence of intervals with $a_i=I  _{i}/h$ for $i=0,1,\cdots,{\cal M}_{n}$.
According to the properties listed above, we have
\begin{equation}
\begin{aligned}
n-1=a_0-a_{{\cal M}_n}=\sum\limits_{i=0}^{{\cal M}_n-1}(a_i-a_{i+1})\leq\sum\limits_{i=0}^{{\cal K}}({\cal N}_{\tau}-1)\big({\cal N}_{\tau}\big)^{i}=\big({\cal N}_{\tau}\big)^{{\cal K}+1}-1,
\end{aligned}
\end{equation}
and the lower bound of  ${\cal M}_n$ is given by ${\cal M}_n\geq ({\cal N}_{\tau}-1)(\log_{{\cal N}_{\tau}} n-1)$. Here ${\cal K}$ is an 
integer such that $\frac{{\cal M}_n}{{\cal N}_{\tau}-1}\in({\cal K},{\cal K}+1]$. On the other hand, we get
\begin{equation}
\begin{aligned}
n-1=a_0-a_{{\cal M}_n}=\sum\limits_{i=0}^{{\cal M}_n-1}(a_i-a_{i+1})\geq (2{\cal N}_{\tau}-2)\sum\limits_{i=0}^{{\cal Y}}\big({\cal N}_{\tau}\big)^{i}
=2[\big({\cal N}_{\tau}\big)^{{\cal Y}+1}-1],
\end{aligned}
\end{equation}
where ${\cal Y}$ is an
integer such that $\frac{{\cal M}_n}{2{\cal N}_{\tau}-2}\in({\cal Y},{\cal Y}+1]$. Then the upper bound of ${\cal M}_n$ is 
given by
\begin{equation}
\begin{aligned}
{\cal M}_n\leq 2({\cal N}_{\tau}-1)\log_{{\cal N}_{\tau}}\Big(\frac{n+1}{2}\Big),
\end{aligned}
\end{equation}
which completes our proof.
\end{proof}

In the FAOM method, we can choose a small integer number ${\cal N}_{\tau}$ to control the storage of memory.
Usually,  ${\cal N}_{\tau}$ is set to be 2 or 3.
As compared with the $L1$ approximation, the FAOM method reduces the storage requirement from ${ O}(N_T)$ to ${ O}({\cal N}_{\tau}
\log _{{\cal N}_{\tau}}N_T)$ and the total
computational cost from ${ O}(N_T^2)$ to ${ O}(N_T{\cal N}_{\tau}
\log _{{\cal N}_{\tau}}N_T)$.
Furthermore, the FAOM method will reduce to the $L1$ approximation while  ${\cal N}_{\tau}>N_T$. However,
the numerical results in the next section show that the convergence order of the FAOM goes to zero while $h\rightarrow 0$. 
To improve the FAOM method,
let us go back to equation (\ref{eq:al-1}), in which the history part is approximated as
\begin{equation}
\begin{aligned}
{\cal I}_h(t_n)\approx\sum\limits_{i=1}^{{\cal M}_n} \frac{1}{I_{i-1}-I_{i}}\int\limits_{I_i}^{I_{i-1}}\big(\Pi_{J,h} u(\tau)\big)'\txn 
d \tau\int\limits_{I_i}^{I_{i-1}}\frac{1}{(t_n-\tau)^{\alpha}}\txn d \tau.
\end{aligned}
\end{equation}
In the above formula, $u'(\tau)$ is approximated by a constant function on the subinterval  $[I_i,I_{i-1}]$.
The error of this approximation is dependent on the length of the subinterval  $[I_i,I_{i-1}]$. Since the length of subinterval  $[I_{{\cal M}_n},I_{{\cal M}_{n-1}}]$
dos not go to zeros as $h \rightarrow 0$, the error of the FAOM method with small ${\cal N}_{\tau}$ may not convergent to  zero as $h\rightarrow 0$.
Actually, we can rewrite equation (\ref{eq:al-1}) as follows
\begin{equation}\label{eq:22-1}
\begin{aligned}
\int\limits_0^{t_n}\frac{ u'(\tau)}{(t_n-\tau)^{\alpha}}\txn d \tau& 
\approx\int\limits_{t_{n-1}}^{t_n}\frac{\big(\Pi_{J,h} u(\tau)\big)'}{(t_n-\tau)^{\alpha}}\txn d \tau+\sum\limits_{i=1}^{{\cal M}_n} 
c_i^0\int\limits_{I_i}^{I_{i-1}}\big(\Pi_{J,h} u(\tau)\big)'\txn 
d \tau,
\end{aligned}
\end{equation}
where $c_i^0=\frac{1}{I_{i-1}-I_i}\int\limits_{I_i}^{I_{i-1}}\frac{1}{(t_n-\tau)^{\alpha}}\txn d \tau$. In equation (\ref{eq:22-1}),
$ u(\tau)$ is approximated by $\Pi_{J,h} u(\tau)$, and the kernel function is replaced by
a constant  $c_i^0$ on the subinterval  $[I_i,I_{i-1}]$.
The difference between the two understandings of the FAOM  is shown in Fig.~\ref{fig:1-1} (c-d)
by a simple example.
As shown in Fig.~\ref{fig:1-1}-(d), the error between the piecewise constant function and the kernel function does not go
to zero as $h \rightarrow 0$. 

To improve the FAOM method, we introduce a more accurate approximation for the kernel function,
which is based on a polynomial approximation of
the special function $\frac 1 {(1-\tau)^{\alpha}}$ on the interval $[-\frac 1 3,\frac 1 3]$.
We denote the new method as the high order fast algorithm with optimum memory 
based on a $K$-th  degree polynomial approximation  (FAOM-P$K$).
Suppose the function $\frac 1 {(1-\tau)^{\alpha}}$ is approximated by a polynomial function
$\sum\limits_{i=0}^K w_i\tau^i$.
Let $\epsilon_K$ be the absolute error of the approximation, which is defined as follows
\begin{equation}\label{eq:216}
\Big|\frac{1}{(1-\tau)^{\alpha}}-\sum\limits_{i=0}^{K}w_i\tau^i\Big|\leq \epsilon_K,\quad \tau\in[-\frac 1 3,\frac 1 3].
\end{equation}
Next, we propose the  FAOM-P$K$ method based on the 
polynomial approximation.  
After replacing $u'(\tau)$ by $\big(\Pi_{J,h}u(\tau)\big)'$,
the Caputo fractional derivative  is approximated as
\begin{equation}\label{eq:FEOM-pk}
\begin{aligned}
^C_0{\cal D}_t^\alpha u(t)
=\frac{1}{\Gamma(1-\alpha)}\int\limits_{t_{n-1}}^{t_{n}}\frac{\big(\Pi_{J,h} u(\tau)\big)'}{(t_n-\tau)^{\alpha}}\txn d \tau+
\frac{1}{\Gamma(1-\alpha)}\int\limits_0^{t_{n-1}}\frac{\big(\Pi_{J,h} u(\tau)\big)'}{(t_n-\tau)^{\alpha}}\txn d \tau+{\cal R}_{J,h},
\end{aligned}
\end{equation}
where ${\cal R}_{J,h}$ is the truncation error according to the polynomial approximation of $ u(\tau)$. 
Similar to the FAOM method, we decompose the second integral in the right hand of the above equation into several parts as follows
\begin{equation}\label{eq:217-0}
\begin{aligned}
\frac1{\Gamma(1-\alpha)}\int\limits_0^{t_{n-1}}\frac{\big(\Pi_{J,h} u(\tau)\big)'}{(t_n-\tau)^{\alpha}}\txn d \tau=\frac1{\Gamma(1-\alpha)}
\sum\limits_{i=0}^{{\cal M}_n-1}\int\limits_{I_{i+1}  }^{I_{i}  }\frac{\big(\Pi_{J,h} u(\tau)\big)'}{(t_n-\tau)^{\alpha}} \txn d \tau.
\end{aligned}
\end{equation}
By setting $I_{i+\frac 1 2}  =(I_{i}  +I_{i+1}  )/2$ and $\bar \tau=\tau- I_{i+\frac 1 2}  $,
the $(i+1)$-th term in the right hand of (\ref{eq:217-0})  can be rewritten as
\begin{equation}\label{eq:217-1}
\begin{aligned}
\frac1{\Gamma(1-\alpha)}\int\limits_{I_{i+1}  }^{I_{i}  }\frac{\big(\Pi_{J,h} u(\tau)\big)'}{(t_n-\tau)^{\alpha}} \txn d \tau=
\frac1{\Gamma(1-\alpha)}
\int\limits_{I_{i+1}  -I_{i+\frac 1 2}  }^{I_{i}  -I_{i+\frac 1 2}  }
\frac{\big(\Pi_{J,h} u(I_{i+\frac 1 2}  +\bar \tau)\big)'}{(t_n-I_{i+\frac 1 2}  -\bar \tau)^{\alpha}} \txn d \bar \tau.
\end{aligned}
\end{equation}
After denoting $\tilde \tau=\frac{\bar \tau}{t_n-I_{i+\frac 1 2}  }$, the kernel function $ (t_n-I_{i+\frac 1 2}  
-\bar \tau)^{-\alpha}$ in (\ref{eq:217-1}) is equal to $(t_n-I_{i+\frac 1 2}  )^{-\alpha}\frac 1 {(1-\tilde \tau)^\alpha} $.
Thanks to Lemma \ref{lem:1}, we have $-\frac 1 3\leq\tilde \tau\leq\frac 1 3$ holds for all $\tau\in[I_{i+1}  ,I_{i}  ]$.
Using the polynomial approximation of the function $\frac{1}{(1-\tilde \tau)^\alpha}$,
the $(i+1)$-th term in the right hand of (\ref{eq:217-0})  can be aprroximated as 
\begin{equation}\label{eq:217-2}
\begin{aligned}
\frac1{\Gamma(1-\alpha)}\int\limits_{I_{i+1}  }^{I_{i}  }\frac{\big(\Pi_{J,h} u(\tau)\big)'}{(t_n-\tau)^{\alpha}} \txn d \tau=
\sum\limits_{k=0}^{K} \frac{\bar w_k^i}{\Gamma(1-\alpha)}\int\limits_{I_{i+1}  }^{I_{i}  }{\big(\Pi_{J,h}u(\tau)\big)'}
\Big(\frac{\tau-I_{i+\frac 1 2}  }{I_{i}  -I_{i+\frac 1 2}  }\Big)^k \txn d \tau+{\cal R}^i_K,
\end{aligned}
\end{equation}
where $\bar w_k^i=w_k\frac{\big(I_i  -I_{i+\frac12}  \big)^k}
{\big(t_n-I_{i+\frac 12}  \big)^{k+\alpha}}$ and ${\cal R}^i_K$ denotes the cut off error according to the polynomial approximation of $\frac 1 {(1-\tilde \tau)^\alpha}$.
By combining (\ref{eq:FEOM-pk}), (\ref{eq:217-0}), and (\ref{eq:217-2}), 
the numerical scheme of the Caputo fractional derivative is finally given as follows
\begin{equation}\label{eq:217}
\begin{aligned}
^C_0{\cal D}_t^\alpha u(t)&=\int\limits_0^{t_{n}}\frac{ u'(\tau)}{(t_n-\tau)^{\alpha}}\txn d \tau=\frac{1}{\Gamma(1-\alpha)}\int\limits_{t_{n-1}}^{t_{n}}\frac{\big(\Pi_{J,h} u(\tau)\big)'}{(t_n-\tau)^{\alpha}}\txn d \tau
\\&+\sum\limits_{i=0}^{{\cal M}_n-1}\sum\limits_{k=0}^{K}\frac{\bar w_k^i}{\Gamma(1-\alpha)}\int\limits_{I_{i+1}  }^{I_{i}  }\big(\Pi_{J,h} u(\tau)\big)'
\Big(\frac{\tau-I_{i+\frac 1 2}  }{I_{i}  -I_{i+\frac 1 2}  }\Big)^k \txn d \tau+{\cal R}_K+{\cal R}_{J,h}\\
&:=^C_0{\cal D}^{F,\alpha}_tu(t)+{\cal R}_K+{\cal R}_{J,h},
\end{aligned}
\end{equation}
where $^C_0{\cal D} ^{F,\alpha}_t u(t)$ presents the numerical Caputo fractional derivative calculated by the FAOM-P$K$ method.
Here 
${\cal R}_K=\sum\limits_{i=0}^{{\cal M}_n-1}{\cal R}_K^i$
denotes the total truncation error according to the polynomial approximation of $\frac 1 {(1-\tilde \tau)^\alpha}$.

In the FAOM-P$K$ method, we save  $\int\limits_{I_{i+1}  }^{I_{i}  }\big(\Pi_{J,h} u(\tau)\big)'
\Big(\frac{\tau-I_{i+1/ 2}  }{I_{i}  -I_{i+1 /2}  }\Big)^k\txn d\tau$, $k=0,1,$ $\cdots,\,K$ at each time step. 
The total memory requirement in the  FEOM-P$K$ method is $
{O}\big((K+1){\cal N}_{\tau}\log_{{\cal N}_{\tau}}n\big)$ at the $n$-th time step.
During the optimizing step, the ${\cal N}_\tau$ subintervals with the same length is combined to a large one.
At the same time, the corresponding integrals we saved also need to be combined.
In the case of $I_{i}  -I_{i+1}  =I_{i+1}  -I_{i+2}  =\cdots=I_{i+{\cal N}_{\tau}-1} 
-I_{i+{\cal N}_{\tau}}  $, the integrals on the subinterval $[I_{i+{\cal N}_\tau},I_{i}]$ can be decomposed as follows
\begin{equation}\label{eq:22-20}
\begin{aligned}
&\int\limits_{I_{i+{\cal N}_{\tau}}  }^{I_{i}  }\big(\Pi_{J,h} u(\tau)\big)'
\Big(\frac{\tau-I_{i+{\cal N}_{\tau}/2}  }{I_{i}  -I_{i+{\cal N}_{\tau}/2}  }\Big)^k \txn d \tau
=\sum\limits_{j=0}^{{\cal N}_{\tau}-1}\int\limits_{I_{i+j+1}  }^{I_{i+j}  }\big(\Pi_{J,h} u(\tau)\big)'
\Big(\frac{\tau-I_{i+{\cal N}_{\tau}/2}  }{I_{i}  -I_{i+{\cal N}_{\tau}/2}  }\Big)^k \txn d \tau
\\&\quad=\sum\limits_{j=0}^{{\cal N}_{\tau}-1}\int\limits_{I_{i+j+1}  }^{I_{i+j}  }\big(\Pi_{J,h} u(\tau)\big)'
\Big(\frac{\tau-I_{i+j+\frac 12}  }{{\cal N}_{\tau}\big(I_{i+j}  -I_{i+j+\frac 1 2}  \big)}
-\frac{2j+1-{\cal N}_{\tau}}{{\cal N}_{\tau}}\Big)^k \txn d \tau\\
&\quad=\sum\limits_{j=0}^{{\cal N}_{\tau}-1}\sum\limits_{l=0}^k\frac{C^l_k}{{\cal N}_{\tau}^k}
\big({\cal N}_{\tau}-2j-1\big)^{k-l}
\int\limits_{I_{i+j+1}  }^{I_{i+j}  }\big(\Pi_{J,h} u(\tau)\big)'
\Big(\frac{\tau-I_{i+j+\frac 12}  }{I_{i+j}  -I_{i+j+\frac 1 2}  }
\Big)^l \txn d \tau,
\end{aligned}
\end{equation}
which gives a rule for optimizing the  storage.  The FAOM-P$K$ algorithm is finally given in Algorithm 2.

\noindent
\begin{minipage}[t]{1.\textwidth}
\noindent
\underline{{~~~~~~~~~~~~~~~~~~~~~~~~~~~~~~~~~~~~~~~~~~~~~~~~~~~~~~~~~~~~~~~~~~~~~~~~~~~~~~~~~~~~~~~~~~~~~~~~~~~~~~~~~~~~~~~}}

\noindent\textbf{Algorithm 2}: FAOM-P$K$
of the Caputo fractional derivative. 
\vspace{-0.3cm}

\noindent\underline{{~~~~~~~~~~~~~~~~~~~~~~~~~~~~~~~~~~~~~~~~~~~~~~~~~~~~~~~~~~~~~~~~~~~~~~~~~~~~~~~~~~~~~~~~~~~~~~~~~~~~~~~~~~~~~~~}}

\vspace{0.15cm}
\noindent\textbf{Initialization:} 
Let $\mathbf I^n=[I^n_1,I^n_2,\cdots,I^n_{{\cal M}_n}]  $ be a vector, whose elements are the starts of subintervals, 
and $\mathbf U^{n,k}=[U^{n,k}_1,U^{n,k}_2,\cdots,U^{n,k}_{{\cal M}_n}]$ with $k=0,1,\cdots,K$, be vectors, whose elements are 
$\int\limits_{I_{i}  }^{I_{i-1}  }\big(\Pi_{J,h} u(\tau)\big)'
\Big(\frac{\tau-I_{i-1/ 2}  }{I_{i-1}  -I_{i-1 /2}  }\Big)^k\txn d\tau$, respectively. Set $\mathbf U^{0,k}=0$
  and $\mathbf I^0  =0$. Pre-chose an integer number ${\cal N}_{\tau}\geq 2$ to control the storage of memory. 
  
\noindent\textbf{Start time loop:} $t_n=nh$ with $n=2,\,3,\cdots,N_T$.
\begin{description}
  \item[~~Step 1] (Updating the storage):  Update the temporary vector 
  $\tilde{\mathbf I}^n  ={\cal B}(t_{n-2},\mathbf I^{n-1}  )$ and the temporary storage vectors by  $\tilde{\mathbf U}^{n,k}={\cal B}\Big(
  \int\limits_{t_{n-2}}^{t_{n-1}}\big(\Pi_{J,h} u(\tau)\big)'
\Big(\frac{\tau-t_{n-3/2}}{h/2}\Big)^k\txn d\tau,\mathbf U^{n-1,k}\Big)$.
    \item[~~Step 2] (Optimizing of the storage):  Obtain the storage vectors $\mathbf{U}^{n,k}$ and vector $\mathbf{I}^n$ as follows.   
  \end{description}
\begin{itemize}
  \item If there exists $i_0$ such that $\tilde I^n  _{i_0}-\tilde I^n  _{i_0+1}=\cdots=
   \tilde I ^n_{i_0+2{\cal N}_{\tau}-2}-\tilde I^n  _{i_0+2{\cal N}_{\tau}-1}$, let 
  $\mathbf U^n  ={\cal F}({\cal N}_{\tau},i_0+{\cal N}_{\tau}, v,
   \tilde{\mathbf U}^n  )$ 
   and $\mathbf I^n  ={\cal F}({\cal N}_{\tau},i_0+{\cal N}_{\tau},
   \tilde{\mathbf I}^n  )$. Here $I ^n _0=\tilde I^n_0=t_{n-1}$
   and $v=\int\limits_{\tilde I_{i_0+2{\cal N}_\tau-1}  }^{\tilde I_{i_0+{\cal N}_\tau-1}  }\big(\Pi_{J,h} u(\tau)\big)'
\Big(\frac{\tau-I_{i-1/ 2}  }{I_{i-1}  -I_{i-1 /2}  }\Big)^k\txn d\tau$ is calculated according to (\ref{eq:22-20}).
   Set $\tilde {\mathbf I}^n=\mathbf{I}^n$, $\tilde {\mathbf U}^{n,k}=\mathbf{U}^{n,k}$ and redo  optimization until there does not exist $i_0$ such that $I^n  _{i_0}-I^n  _{i_0+1}=\cdots=
   I^n  _{i_0+2{\cal N}_{\tau}-2}-I^n  _{i_0+2{\cal N}_{\tau}-1}$ are satisfied.
  \end{itemize}
  \begin{description}
  \item[~~Step 3] (Calculating the Caputo fractional derivative):  
   Obtain the numerical Caputo fractional derivative according to (\ref{eq:217}).
  \end{description}

\noindent\textbf{End of time loop}.

\vspace{-0.2cm}
\noindent\underline{{~~~~~~~~~~~~~~~~~~~~~~~~~~~~~~~~~~~~~~~~~~~~~~~~~~~~~~~~~~~~~~~~~~~~~~~~~~~~~~~~~~~~~~~~~~~~~~~~~~~~~~~~~~~~~~~}}
\vspace{0.2cm}
\end{minipage}

The truncation error of the FAOM-P$K$ algorithm can be decomposed into two parts,  which are 
${\cal R}_K$ and ${\cal R}_{J,h}$. 
Here ${\cal R}_K$  is the total truncation
error according to the polynomial approximation of the kernel function
and ${\cal R}_{J,h}$ is the truncation error according to the polynomial interpolation of $u(t)$.
It is worth to pointing out that the two parts
are independent. The estimate of the truncation error ${\cal R}_{J,h}$, according to the $L$1 or $L1-2$ polynomial interpolation of $ u(t)$, 
can be founded in  \cite{gao14, sun06}.
The following Lemma, which can be found in \cite{sun06},
establishes an error bound for the $L$1 formula.
\begin{lem}
Suppose $u(t)\in C^2[0,t_{n}]$. For any $\alpha$ $(0<\alpha<1)$, then
\begin{equation}
\begin{aligned}
\big|{\cal R}_{1,h}\big|&
&\leq \frac{1}{1-\alpha}\Big[\frac{1-\alpha}{12}+\frac{2^{2-\alpha}}{2-\alpha}-(1+2^{-\alpha})\Big]
\max\limits_{0\leq t\leq t_{n}}\big|u''(t)\big|h^{2-\alpha}.
\end{aligned}
\end{equation}
\end{lem}
The truncation error of $L1-2$ formula  is illustrated in the following Lemma, which can be found in \cite{gao14}.
\begin{lem}
Suppose $u(t)\in C^3[0,t_{n}]$. For any $\alpha$ $(0<\alpha<1)$,  then
\begin{equation}
\big|{\cal R}_{2,h}\big|\leq\left\{
\begin{aligned}
&\frac{\alpha}{2\Gamma(3-\alpha)}\max\limits_{0\leq t\leq t_1}\big| u''(t) \big|h^{2-\alpha}, \quad n=1,\\
&\frac1{\Gamma(1-\alpha)}\Big\{\frac\alpha{12}\max\limits_{0\leq t\leq t_1}\big| u''(t) \big|(t_n-t_1)^{-\alpha-1}h^3+\Big[\frac 1 {12}\\
&\quad+\frac{\alpha}{3(1-\alpha)(2-\alpha)}\Big(\frac 1 2+\frac1{3-\alpha}\Big)\Big]
\max\limits_{0\leq t\leq t_n}\big| u'''(t) \big|h^{3-\alpha}\Big\},\quad n\geq2.
\end{aligned}
\right.
\end{equation}
\end{lem}

The following two Lemmas establish an error bound for ${\cal R}_K$.

\begin{lem}\label{lem:5}
At the $n$-th time step, the following inequality is satisfied
\begin{equation}
\bigg|\Big(\frac{{t_n-I_{i+\frac 1 2}  }}{t_n-\tau}\Big)^\alpha-\sum\limits_{k=0}^Kw_k\Big(
\frac{\tau-I_{i+\frac 1 2}  }{t_n-I_{i+\frac 1 2}  }\Big)^k
\bigg|\leq \epsilon_K,\quad \forall \tau\in[I_{i+1}  ,I_{i}  ].
\end{equation}
\end{lem}
\begin{proof}
From Lemma \ref{lem:1}, we have $\frac{\tau-I_{i+1}  }{t_n-I_{i+1}  }\leq \frac 1 2$ holds for all $\tau\in[I_{i+1}  ,
I_{i}  ]$, which implies $t_n-I_{i+\frac12}  \geq 3 (
I_{i}  -I_{i+\frac 1 2}  )$.
Then we obtain $\Big|
\frac{\tau-I_{i+\frac 1 2}  }{t_n-I_{i+\frac 1 2}  }\Big|\leq \frac 1 3$ holds for all
$\tau\in[I_{i+1}  ,I_{i}  ]$.
Thanks to (\ref{eq:216}), the following inequality holds for all $\tau\in[I_{i+1}  ,I_{i}  ]$
\begin{equation}
\begin{aligned}
\bigg|\Big(\frac{{t_n-I_{i+\frac 1 2}  }}{t_n-\tau}\Big)^\alpha-\sum\limits_{k=0}^Kw_k\bar \tau^k
\bigg|
=\Big|\frac{1}{(1-\bar \tau)^\alpha}-\sum\limits_{k=0}^Kw_k\bar \tau^k
\Big|\leq \epsilon_K,
\end{aligned}
\end{equation}
where $\bar \tau=
\frac{\tau-I_{i+\frac 1 2}  }{t_n-I_{i+\frac 1 2}  }$. The proof of the Lemma is completed.

\end{proof}

\begin{lem}\label{lem:6}
Suppose $u(t)\in C^1[0,t_n]$. At the $n$-th time step, the total cut off error according to the polynomial
approximation ${\cal R}_K$ is bounded by
\begin{equation}
\begin{aligned}
\big|{\cal R}_K\big|\leq C\epsilon_Kt_n^{1-\alpha}
\max\limits_{0\leq t\leq t_n}\big|u'(t)\big|,
\end{aligned}
\end{equation}
where $C$ is a constant independent of $h$ and $\epsilon_K$.
\end{lem}
\begin{proof}
From Lemma \ref{lem:5} and (\ref{eq:217-2}),  we have 
\begin{equation}
\begin{aligned}
\big|{\cal R}^i_K\big|&=\frac1{\Gamma(1-\alpha)}\Bigg|\int\limits_{I_{i+1}  }^{I_{i}  }\frac{\big(\Pi_{J,h} u(\tau)\big)'}
{\big(t_n-I_{i+\frac 12}  \big)^\alpha}\bigg(\Big(\frac{t_n-I_{i+\frac 12}  }{t_n-\tau}\Big)^{\alpha} -
\sum\limits_{k=0}^{K} w_k^i
\Big(\frac{\tau-I_{i+\frac 1 2}  }{t_n-I_{i+\frac 12}  }\Big)^k\bigg) \txn d \tau\Bigg|\\
&\leq C\epsilon_K
\frac{{I_{i}  }-{I_{i+1}  }}
{\big(t_n-I_{i+\frac 12}  \big)^\alpha}\max\limits_{I_{i+1}  \leq t\leq I_{i}  }\big|u'(t)\big|\\
&\leq C\epsilon_K
({{I_{i}  }-{I_{i+1}  }})^{1-\alpha}\max\limits_{I_{i+1}  \leq t\leq I_{i}  }\big|u'(t)\big|,
\end{aligned}
\end{equation}
where $C$ is a constant independent of $h$ and $\epsilon_K$.
By taking the summation of the above equations from $0$ to ${\cal M}_n-1$, it follows that the bound of the total truncation error according to the polynomial
approximation ${\cal R}_K$ is given by
\begin{equation}
\begin{aligned}
\big|{\cal R}_K\big|\leq C\epsilon_Kt_n^{1-\alpha}
\max\limits_{0\leq t\leq t_n}\big|u'(t)\big|.
\end{aligned}
\end{equation}
Here $C$ is a constant independent of $h$ and $\epsilon_K$.
\end{proof}

\begin{niao}\label{thm:1}
Suppose $u(t)\in C^{J+1}[0,t_n]$. For any $\alpha~(0<\alpha<1)$, the gap between the Caputo fractional derivative
$^C_0{\cal D} ^{\alpha}_tu(t)$ and the numerical Caputo fractional derivative $^C_0{\cal D} ^{F,\alpha}_tu(t)$ satisfies 
\begin{equation}
\Big|~^C_0{\cal D} ^{\alpha}_tu(t)-^C_0{\cal D} ^{F,\alpha}_tu(t)\Big|\leq {\cal R}_{J,h}+C\epsilon_Kt_n^{1-\alpha}\max\limits_{0\leq t\leq t_n}\big|u'(t)\big|,
\end{equation}
where  $C$ is a constant independent of $h$ and $\epsilon_K$.
\end{niao}

By combining the above Lemmas, we can prove Theorem \ref{thm:1} easily. As shown in Theorem \ref{thm:1}, the 
FAOM-P$K$ method should have the same order of convergence rate as the corresponding direct method with  small $\epsilon_K$.
Numerical results in the next two sections show that $\epsilon_K\approx $1e-3 with $K=4$ is good enough corresponding to
$L1$ approach and  $\epsilon_K\approx $1e-6 with $K=9$ is good enough corresponding to
$L1-2$ formula, respectively.
From Theorem \ref{thm:1}, we conclude that the fast algorithm FAOM-P$K$ can be used together with any direct scheme with polynomial interpolation of $u(t)$.

\section{Validity of the proposed methods}
In this section, we validate the proposed methods and study the convergence rates. 
Let us consider the following pure initial value problem of the linear fractional diffusion equation
\begin{equation}\label{eq:IVP}
\begin{aligned}
&^C_0{\cal D}^\alpha_t u(x,t)=u_{xx}(x,t)+f(x,t)\qquad &x\in \Omega_x, ~t>0,\\
&u(x,0)=u_0(x)\qquad & x\in\Omega_x,\\
&u(x,t)=\varphi(x,t)\qquad & x\in \partial \Omega_x,~t>0,
\end{aligned}
\end{equation}
where the 
domain $\Omega_x=[a,b]$. Suppose the domain $\Omega_x$ is covered by  a uniform mesh 
with $\Delta x=\frac{b-a}{N}$. The set of all mesh points is denoted as $\Omega_x^{\Delta x}=\big\{x_i,i=0,\,1,\,\cdots,\,N\big\}$,
with $x_i=a+i\Delta x$. We will simply denote the approximation of $u(x_i,t_j)$ by $u^j_i$. 
In this section, we consider a second-order and a fourth-order finite difference scheme to discretize the  spatial
derivative $u_{xx}$, respectively. In the second-order finite difference scheme, $u_{xx}$ is discretized by the central scheme as follows
\begin{equation}\label{eq:central}
u_{xx}(x_i,t_j)\approx \frac{u^j_{i+1}+u^j_{i-1}-2u^j_i}{\Delta x^2}.
\end{equation}
The fourth-order finite difference scheme we used is proposed in \cite{gao13}, which is a compact difference scheme
and achieves fourth-order accuracy in space. In this section, a test case is studied to validate the proposed methods.

\noindent\textbf{Example 3.1}. In (\ref{eq:IVP}), we set the computational domain $\Omega_x=[0,\pi]$. The source term $f(x,t)$, the initial data $u_0(x)$,
and the boundary value $\varphi(x,t)$ are given by
\begin{equation}
\begin{aligned}f(x,t)=
&\Gamma(4+\alpha)x^4(\pi-x)^4\exp(-x)t^3/6-x^2(\pi-x)^2\big\{t^{3+\alpha}\exp(-x)\\
&\quad [x^2(56-16x+x^2)-2\pi x(28-12x+x^2)+\pi^2(12-8x+x^2)]\\
&\quad +4(3\pi^2-14\pi x+14x^2)\big\}  \quad\qquad x\in\Omega_x,~t\in(0,T],\\
u_0(x)~=&
x^4(\pi-x)^4 \quad\qquad\qquad\qquad\qquad\quad~~ x\in\Omega_x,\\
\varphi(x,t)=&x^4(\pi-x)^4\big[\exp(-x)t^{3+\alpha}+1\big]\qquad~  x\in\partial\Omega_x,~t\in(0,T].
\end{aligned}
\end{equation}
It is clear that the linear problem (\ref{eq:IVP}) has the following exact solution
\begin{equation}
u(x,t)=x^4(\pi-x)^4\big[\exp(-x)t^{3+\alpha}+1\big]\qquad x\in \Omega_x,~t\in[0,T].
\end{equation}
To test the accuracy of our schemes, we define the maximum norm of the error and the convergence rates with respect
to temporal and spatial mesh sizes, which are given as follows
$$
E(\Delta x,h)=\sqrt{h\sum\limits_{j=1}^{N_T}\|\textbf{e}^j\|_\infty^2},\quad r_s=\log_2\frac{E(\Delta x,h)}{E(\Delta x/2,h)},
\quad r_t=\log_2\frac{E(\Delta x,h)}{E(\Delta x,h/2)},
$$
where the error $e^j_i=u(x_i,t_j)-u_i^j$. 

To understand the accuracy of the FAOM method and the FAOM-P$K$ method in time, 
we run the code with different time step sizes $h=1/10,\,1/20,\,1/40,\,1/80$, $1/160$, and a fixed spatial mesh size $\Delta x=\pi/20000$. 
For comparison, we also simulate the example by the cut off approach, $L1$ formula, $L1-2$ formula, and
a fast method proposed in \cite{jiang}. The computational errors and numerical convergence rates for different methods
with $\alpha=0.9,\,0.5,\,0.1$
are given in Tables \ref{Table:0-1} and \ref{Table:0-2}. The results reported in Table \ref{Table:0-1} show that the error of cut off 
approach increases as the time step size decrease. However, the FAOM method achieves  bounded errors for all time
step sizes we used, even the storages of memory for the two methods are almost the same. As reported in the tables, the $L1$ 
formula and $L1-2$ formula both reach the ideal convergence orders, which are $2-\alpha$ and $3-\alpha$, respectively.
We also find that the accuracy of the FAOM-P$K$ method is as good as the $L1$ formula and $L1-2$ formula, if
the same interpolation function $\Pi_{J,h}u(t)$ is used. 
All results are consistent with our analysis given in the previous section.
In the paper, the polynomial approximation of the kernel function $\frac{1}{(1-\tau)^\alpha}$ is given by the Taylor expansion.

\begin{table}[htb]\caption{The errors and convergence orders in time with fixed spatial mesh size
$\Delta x=\pi/20000$ for the proposed methods. In all methods, $u(t)$ is approximated by $\Pi_{1,h}u(t)$,
and the second-order finite difference scheme is used.
$T=1$, ${\cal N}_{\tau}=2$, $\bar S=10$. In the FAOM-P$K$ method, we choose $K=4$
such that $\epsilon_K\approx 1$e-3. Here ``JIANG" denotes the fast algorithm presented in \cite{jiang}.
}\label{Table:0-1}
\begin{center}
\begin{tabular}{|c|cc|cc|cc|cc|cc|}\hline
 \multirow{2}*{$h$}
  &   \multicolumn{2}{|c|}{Cut off} &   \multicolumn{2}{|c|}{FAOM}&  \multicolumn{2}{|c|}{$L1$ formula}&   \multicolumn{2}{|c|}{JIANG} &   \multicolumn{2}{|c|}{FAOM-P$4$}  \\
   & $E(\cdot,\cdot)$  & $r_t$ & $E(\cdot,\cdot)$  & $r_t$& $E(\cdot,\cdot)$  & $r_t$& $E(\cdot,\cdot)$  & $r_t$& $E(\cdot,\cdot)$  & $r_t$\\ 
 \hline
      \multicolumn{11}{|c|}{$\alpha=0.9$} \\ 
 \hline
 1/10   &  3.66e-1	  & - & 3.69e-1  & 1.16& 3.66e-1  & 1.17& 3.66e-1  & 1.17& 3.66e-1  & 1.17\\ 
 1/20   & 1.66e-1 	  & - & 1.65e-1  & 1.11& 1.62e-1  & 1.14& 1.62e-1  & 1.14& 1.62e-1  & 1.14\\ 
 1/40   & 1.06e-1 	  & - & 7.66e-2  & 1.05& 7.39e-2  & 1.12& 7.39e-2 & 1.12& 7.39e-2  & 1.12\\ 
 1/80   & 1.34e-1         & - & 3.69e-2  & 0.97& 3.41e-2  & 1.11& 3.41e-2  & 1.11& 3.41e-2  & 1.11\\ 
 1/160 & 2.16e-1 	  & - & 1.89e-2  & -     & 1.58e-2     & -& 1.58e-2       & - & 1.58e-2  & -\\ 
 \hline
       \multicolumn{11}{|c|}{$\alpha=0.5$} \\ 
 \hline
 1/10   &  7.59e-2	  & - & 8.22e-2  & 1.28& 7.59e-2  & 1.48& 7.60e-2  & 1.48& 7.60e-2  & 1.48\\ 
 1/20   & 6.10e-2 	  & - & 3.38e-2  & 1.02& 2.73e-2  & 1.47& 2.73e-2  &1.47& 2.73e-2  & 1.47\\ 
 1/40   & 2.45e-1 	  & - & 1.67e-2  & 0.63& 9.83e-3  & 1.48& 9.84e-3  & 1.47& 9.85e-3  & 1.47\\ 
 1/80   &   6.11e-1       & - & 1.08e-2  & 0.27& 3.54e-3  & 1.48& 3.54e-3  & 1.48& 3.56e-3  & 1.47\\ 
 1/160 & 1.11 		  & - & 8.97e-3  & -     & 1.27e-3  & -& 1.27e-3  & -     & 1.29e-3  & -\\ 
 \hline
        \multicolumn{11}{|c|}{$\alpha=0.1$} \\ 
 \hline
 1/10   &  5.35e-3	  & - & 7.50e-3  & 1.03& 5.35e-3  & 1.76& 5.35e-2  & 1.76& 5.35e-2  & 1.76\\ 
 1/20   & 1.10e-1 	  & - & 3.68e-3  & 0.46& 1.58e-3  & 1.77& 1.58e-3  & 1.77& 1.58e-3  & 1.76\\ 
 1/40   & 4.91e-1 	  & - & 2.67e-3  & 0.11& 4.63e-4  & 1.78& 4.64e-4  & 1.77& 4.65e-4  & 1.77\\ 
 1/80   & 1.05              & - & 2.47e-3  & -      & 1.35e-4  & 1.79& 1.36e-4  & 1.77& 1.36e-4  & 1.76\\ 
 1/160 & 1.61 		  & - & 2.48e-3  & -      & 3.90e-5  & -& 3.99e-5  & -      & 4.00e-5  & -\\ 
 \hline
\end{tabular}
\end{center}
\end{table}

\begin{table}[htb]\caption{The errors and convergence orders in time with fixed spatial mesh size
$\Delta x=\pi/20000$ for the proposed methods. In all methods, $u(t)$ is approximated by $\Pi_{2,h}u(t)$,
and the fourth-order compact finite difference scheme is used.
$T=1$, ${\cal N}_{\tau}=2$. In the FAOM-P$K$ method, we choose $K=9$
such that $\epsilon_K\approx 1$e-6.
}\label{Table:0-2}
\begin{center}
\begin{tabular}{|c|c|cc|cc|}\hline
  \multirow{2}*{}&\multirow{2}*{$h$}
  &  \multicolumn{2}{|c|}{$L1-2$}&      \multicolumn{2}{|c|}{FAOM-P$9$}  \\
&   & $E(\cdot,\cdot)$  & $r_t$& $E(\cdot,\cdot)$  & $r_t$\\ 
 \hline
  \multirow{5}*{$\alpha=0.9$}
 &1/10      & 6.30e-2  & 2.06& 6.30e-2  & 2.06\\ 
 &1/20      & 1.51e-2  & 2.08& 1.51e-2  & 2.08\\ 
 &1/40      & 3.57e-3  & 2.09& 3.57e-3  & 2.09\\ 
 &1/80     & 8.39e-4  & 2.09& 8.39e-4  & 2.09\\ 
 &1/160  & 1.96e-4  & -   &  1.96e-2  & -\\ 
 \hline \multirow{5}*{$\alpha=0.5$}
 &1/10    &  1.03e-2  & 2.44& 1.02e-2  & 2.44\\ 
 &1/20    &  1.89e-3 & 2.46 & 1.89e-3  & 2.46\\ 
 &1/40    &  3.44e-4  & 2.47& 3.44e-4  & 2.47\\ 
 &1/80    &  6.21e-5  & 2.48& 6.21e-5  & 2.48\\ 
 &1/160   &  1.11e-5  & -      & 1.11e-5  & -\\ 
 \hline
  \multirow{5}*{$\alpha=0.1$}
 &1/10   &  5.59e-4  & 2.76& 5.54e-4  & 2.76\\ 
 &1/20   &  8.24e-5  & 2.77& 8.20e-5  & 2.77\\ 
 &1/40   &  1.20e-5  & 2.94& 1.20e-5  & 2.94\\ 
 &1/80   &  1.57e-6  & 2.45& 1.57e-6  & 2.45\\ 
 &1/160 &  2.88e-7  & -      & 2.88e-7 & -\\ 
 \hline
\end{tabular}
\end{center}
\end{table}

To check the convergence rate of our FAOM-P$K$ method in space,
we do the simulations with different spatial mesh sizes $\Delta x=\pi/20,\,\pi/40,\,\pi/80,\,\pi/160$, $\pi/320$
and a fixed time step size $h=0.001$. Two simulations with $\alpha=0.5$ are considered. In the first one, $u(t)$ is approximated by
$\Pi_{1,h}u(t)$ and $u_{xx}$ is discretized by the second-order finite difference scheme. 
In the second simulation, $u(t)$ is approximated by
$\Pi_{2,h}u(t)$ and $u_{xx}$ is discretized by the fourth-order compact finite difference scheme. As shown in Fig.~\ref{fig:1-4},
the FAOM-P$K$ can reach the ideal convergence order in space for the two finite difference schemes. 

We then investigate the long time performance of 
FAOM-P$K$ method. We compute the example until $T=10$ with $h=0.01$ and $\Delta x=\pi/20000$. Four different methods are considered, which are
$L1$ formula, $L1-2$ formula, the fast method proposed in \cite{jiang}, and the FAOM-P$K$ method, respectively. We focus
on the accuracy and memory usage of those methods. As plotted in Fig.~\ref{fig:1-5}-(a), the the FAOM-P$K$ method and the fast method proposed in \cite{jiang}
can reach the same accuracy with the $L1$ formula.
Furthermore, the high order FAOM-P$K$ algorithm
also has similar accuracy as compared with the $L1-2$ formula. 
We plot the length of vector $\mathbf{U}^n$ as a function of $n$ in Fig.~\ref{fig:1-5}-(b),
from which we clearly see that ${\cal M}_n$ is between $\log_2 n-1$ and $2\log_2\frac{n+1}{2}$. The relationship verifies the 
Lemma \ref{lem:2.2} and shows that the storage of memory for the FAOM-P$K$ algorithm is $
{O}\big((K+1){\cal N}_{\tau}\log_{{\cal N}_{\tau}}n\big)$ at
each time step. 
As a function of total time steps $N_T$, the total computational times of the direct methods and the FAOM-P$K$ method are plotted
in Fig.~\ref{fig:1-5}-(c-d). We observe that the total compute time increases almost linearly
with the total number of time steps $N_T$  for the FAOM-P$K$ method, but the total compute time for the direct scheme 
is in the order of $O(N_T^2)$. There is a significant speed-up in the FAOM-P$K$ algorithm as compared with the direct schemes.

\begin{figure}[htb!]
\begin{center}
{\includegraphics[width=0.47\textwidth]{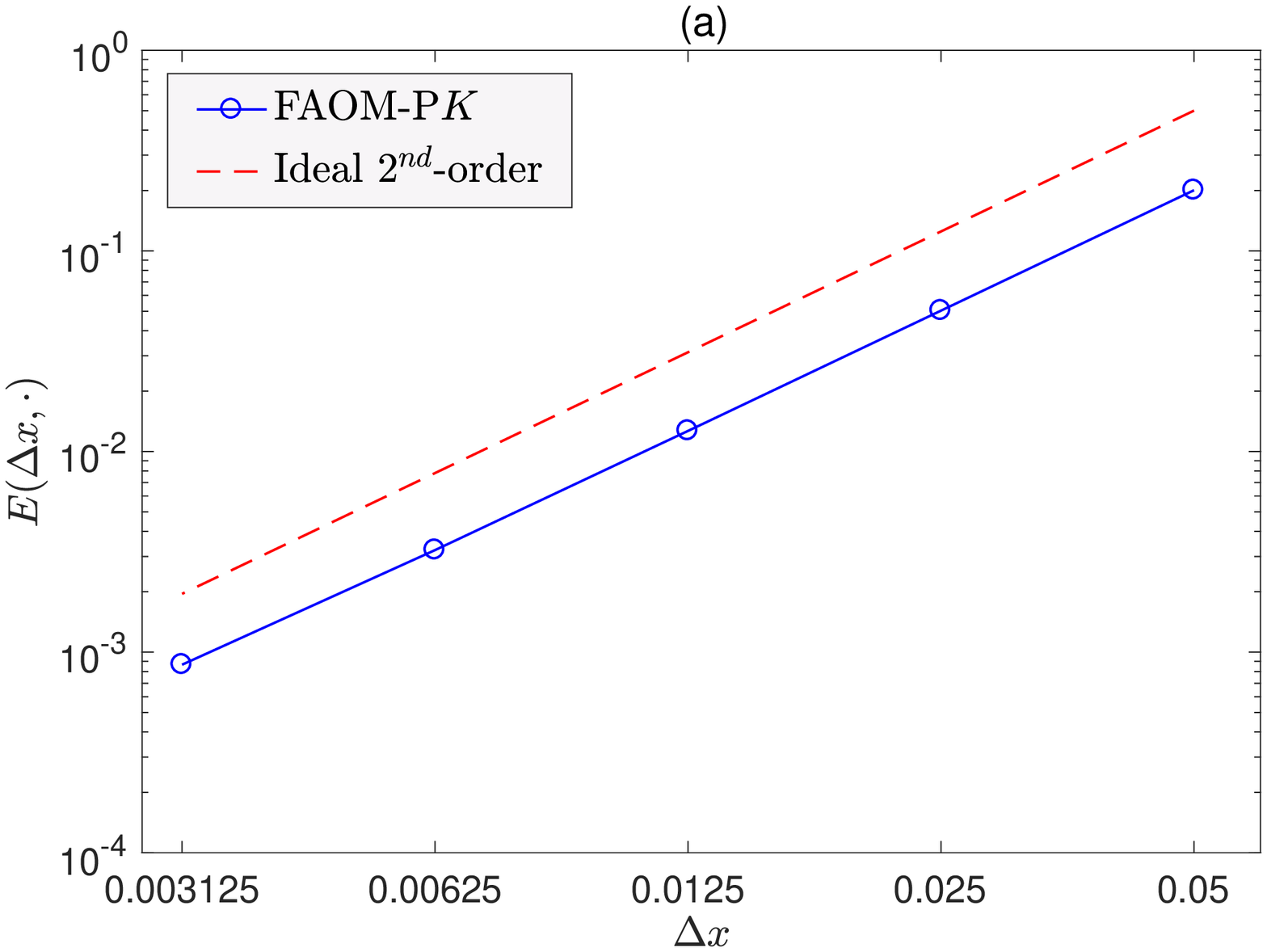}}
{\includegraphics[width=0.47\textwidth]{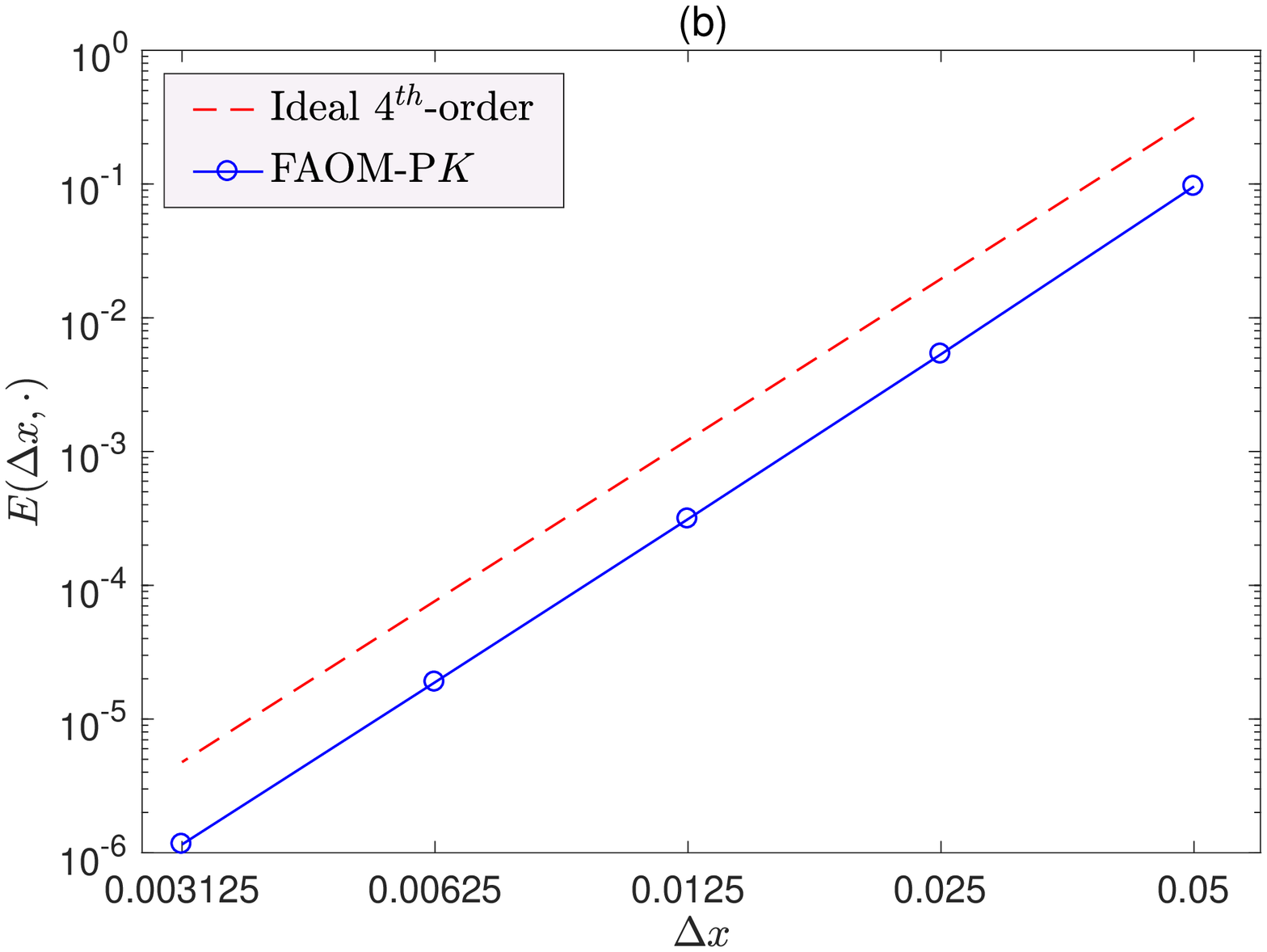}}
\end{center}
\caption{
The errors in space with fixed time step size
$h=0.001$ for the proposed methods and $T=1$. (a) 
$u(t)$ is approximated by $\Pi_{1,h}u(t)$, the second-order finite difference scheme is
used to discretize $u_{xx}$,  and $K=4$; (b) $u(t)$ is approximated by $\Pi_{2,h}u(t)$, the fourth-order compact finite difference scheme is
used to discretize $u_{xx}$, and $K=9$.
  }\label{fig:1-4}
\end{figure}

\begin{figure}[htb!]
\begin{center}
{\includegraphics[width=0.47\textwidth]{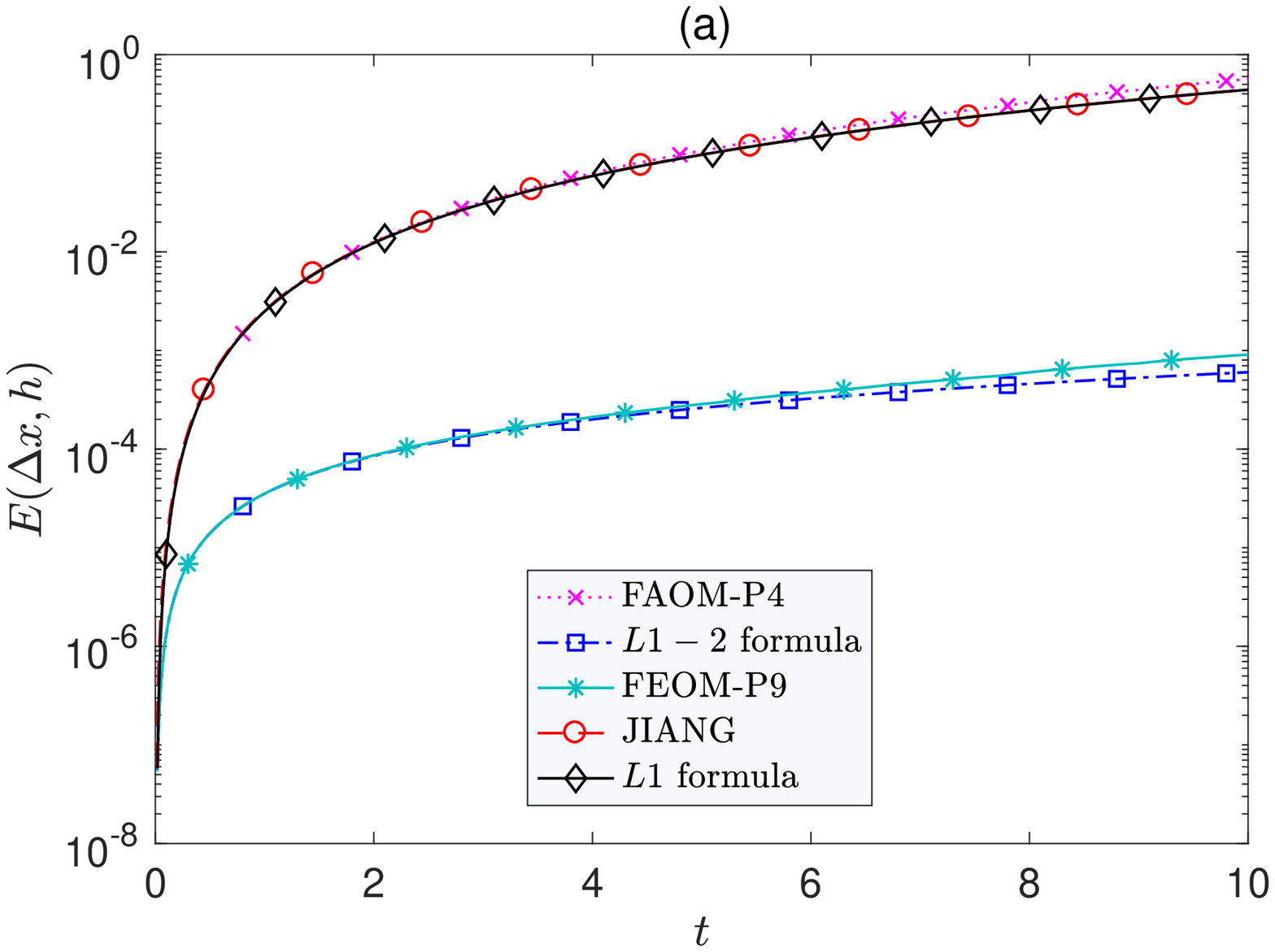}}
{\includegraphics[width=0.47\textwidth]{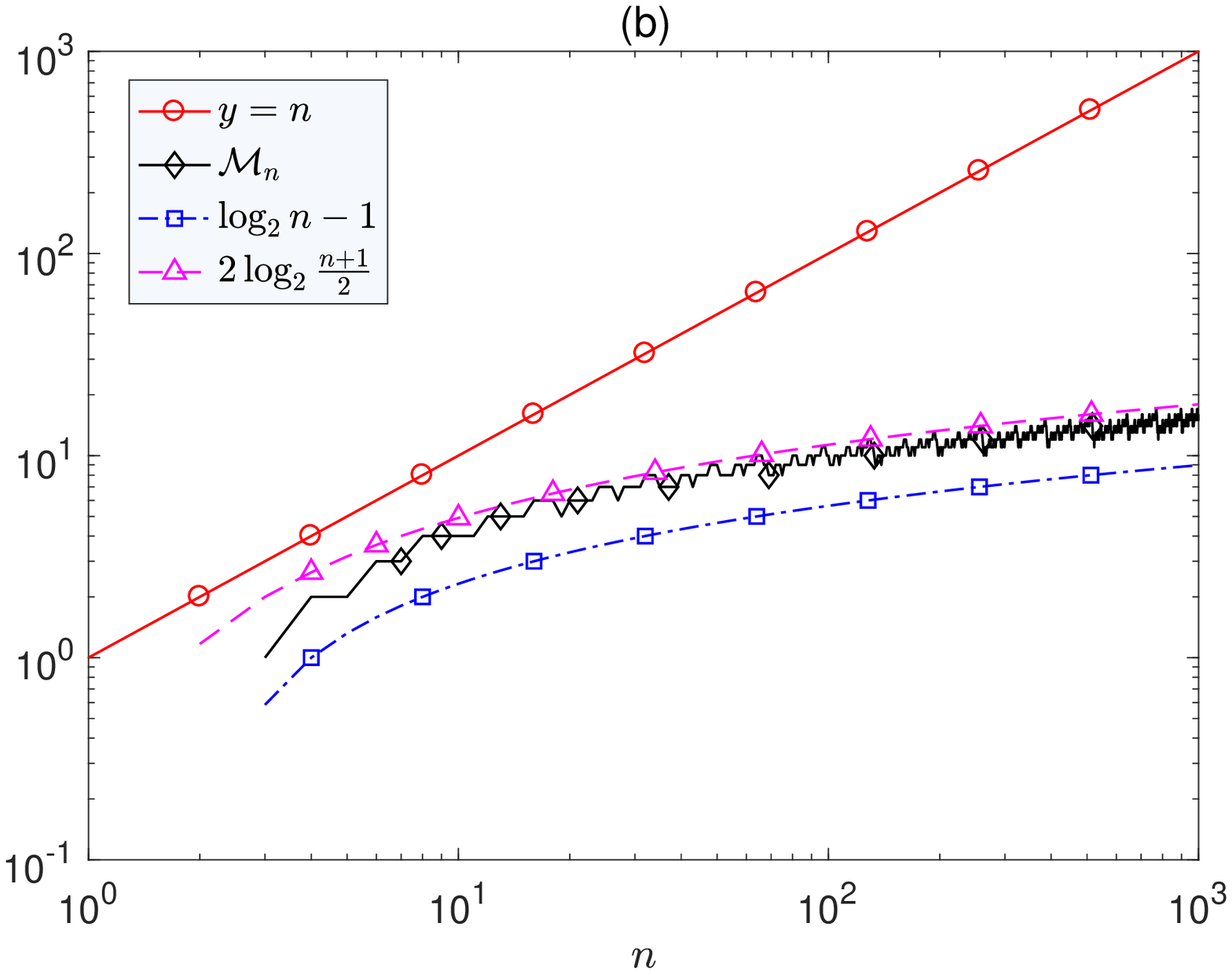}}
{\includegraphics[width=0.47\textwidth]{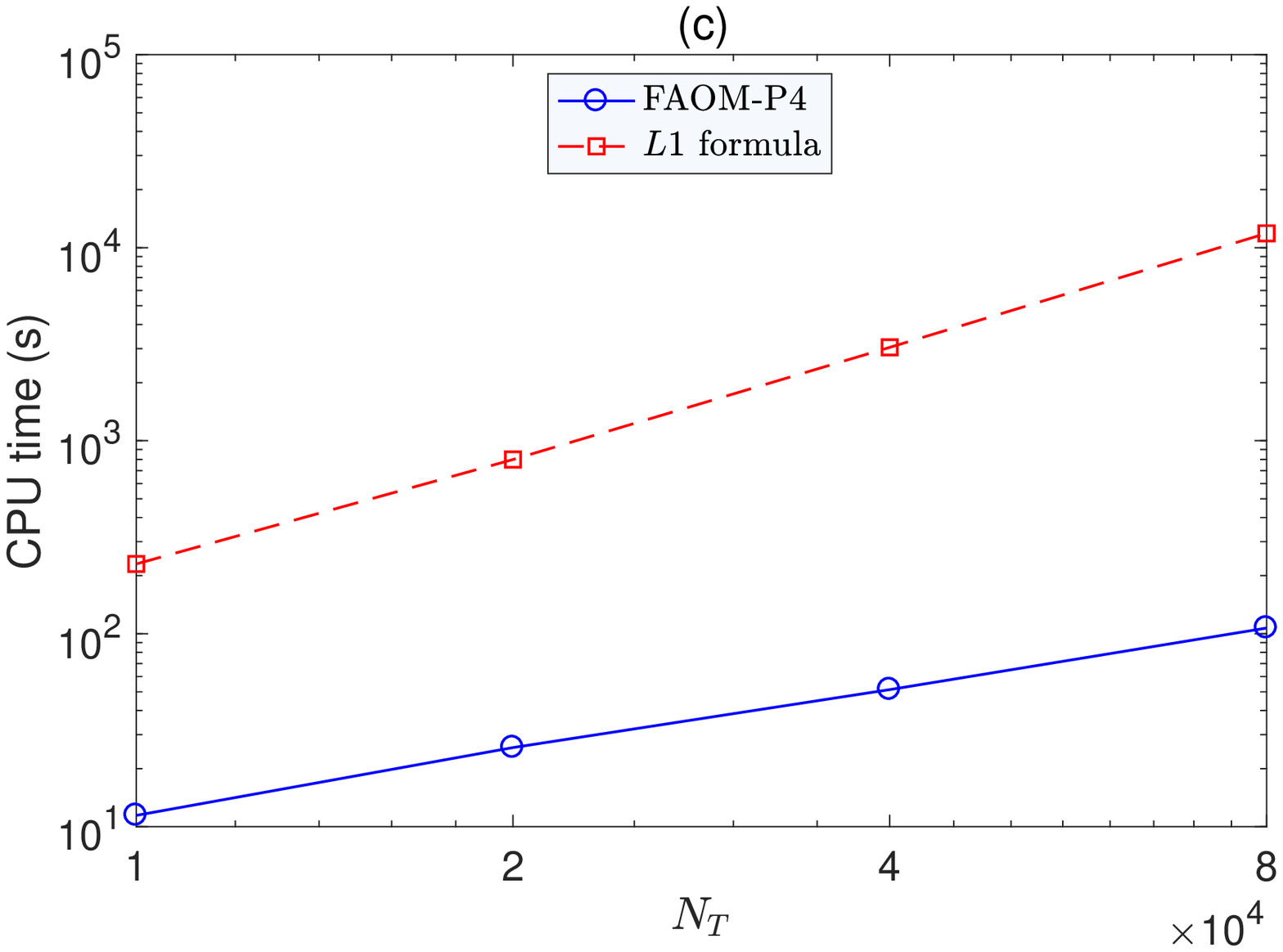}}
{\includegraphics[width=0.47\textwidth]{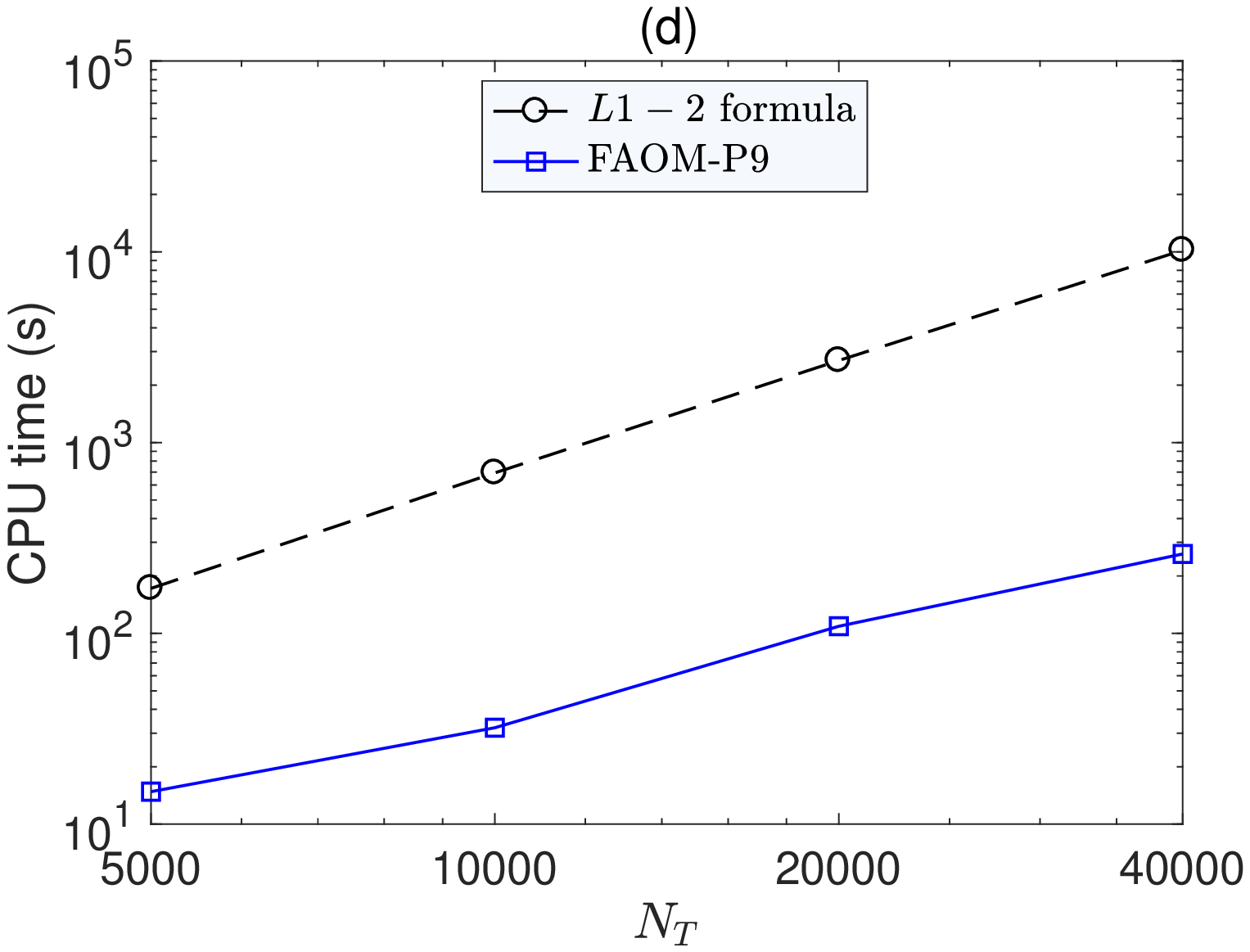}}
\end{center}
\caption{
Long time performance of the proposed methods. $\alpha=0.5$. Here FEOM-P4 denotes the FAOM-P$K$ method with $\Pi_{1,h}u(t)$ 
interpolation approximation and an approximation of $\frac 1 {(1-\tau)^\alpha}$ by a polynomial of the
fourth degree. 
FAOM-P9 denotes the FAOM-P$K$ method with $\Pi_{2,h}u(t)$ 
interpolation approximation and an approximation of $\frac 1 {(1-\tau)^\alpha}$ by a polynomial of the
ninth degree. (a) The error $E(\Delta x,h)$ at each time step for the given methods are given.
(b) The relationship between ${\cal M}_n$ (the length of the vector $\mathbf{U}^{n,k}$) and $n$ at each time step.
(c) The total computational times for the FAOM-P4 algorithm and $L1$ formula with $N=21$ and $\alpha=0.5$.
(d) The total computational times for the FAOM-P9 algorithm and $L1-2$ formula with $N=21$ and $\alpha=0.5$.
  }\label{fig:1-5}
\end{figure}

\section{Nonlinear fractional diffusion equation}
We now consider the initial value problem of the nonlinear fractional diffusion equation as follows 
\begin{equation}\label{eq:IVP-N}
\begin{aligned}
&^C_0{\cal D}^\alpha_t u(x,t)=u_{xx}(x,t)+f(u)+g(x,t)\qquad &x\in \Omega_x, ~t>0,\\
&u(x,0)=u_0(x)\qquad & x\in\Omega_x,\\
&u(x,t)=\varphi(x,t)\qquad & x\in \partial \Omega_x,~t>0.
\end{aligned}
\end{equation}
In this section, we focus on the discretization of the fractional diffusion derivative and discretize the spatial derivative $u_{xx}$ by the second-order finite difference
scheme given in (\ref{eq:central}). To complete the discretization of the nonlinear fractional diffusion equation, 
it still needs to consider the approximation of $f(u)$. If we treat this term implicitly, a nonlinear algebraic system is constructed and 
needs to be solved  at each time step.
This may lead extra computational cost and make the algorithm complicated. In \cite{jiang, lidf16}, $f(u(x_i,t_{j+1}))$ is explicitly approximated as $f(u_i^j)$ 
at the $(j+1)$-th time step. This explicit approximation is high efficient and easily implemented. However, the method only
enjoys a first-order accuracy in time even $L1$ formula is used \cite{jiang, lidf16}. 
This is because the accuracy of the approximation $u^{j+1}\approx u^j$ is only first-order. In this paper, we treat $f(u)$ explicitly 
with a high order approximation of $u^{j+1}$ by the solution of the 
previous several time steps. At the $(j+1)$-th time step, the discrete scheme for nonlinear fractional diffusion 
equation (\ref{eq:IVP-N}) is given as follows 
\begin{equation}\label{eq:4.2}
\begin{aligned}
&^C_0{\cal D}^{F,\alpha}_t u_i^{j+1}=\frac{u^{j+1}_{i+1}+u^{j+1}_{i-1}-2u^{j+1}_i}{\Delta x^2}+f(\tilde {u}^{j+1}_i)+g(x_i,t_{j+1})&1\leq i\leq N-1,\\
&u_i^0=u_0(x_i)& 0\leq i\leq N,\\
&u_i^{j+1}=\varphi(x_{i},t_{j+1})&i=0~\txn{or}~i=N,
\end{aligned}
\end{equation}
where $\tilde {u}^{j+1}_i$ is a high order approximation of ${u}^{j+1}_i$. In this paper, $\tilde {u}^{j+1}_i$  equals  $2u^{j}_i-u^{j-1}_i$ with $j>0$ 
and $\tilde {u}^{1}_i=u_i^0$ for $L1$ approach. For $L1-2$ approach,
$\tilde {u}^{j+1}_i$  equals  $3u_i^j-3u_i^{j-1}+u_i^{j-2}$ for $j>1$, $\tilde {u}^{2}_i=2u^{1}_i-u^{0}_i$, and
$\tilde {u}^{1}_i=u_i^0$.

\noindent\textbf{Example 4.1}. In (\ref{eq:IVP-N}), we assume the computational domain $\Omega_x=[0,\pi]$. The nonlinear term $f(u)$, the source term $g(x,t)$, the initial data $u_0(x)$,
and the boundary value $\varphi(x,t)$ are given by
\begin{equation}
\begin{aligned}
f(u)~~=&0.01u(1-u),\\g(x,t)=
&\Gamma(4+\alpha)x^4(\pi-x)^4\exp(-x)t^3/6-x^2(\pi-x)^2\big\{t^{3+\alpha}\exp(-x)\\
&\quad [x^2(56-16x+x^2)-2\pi x(28-12x+x^2)+\pi^2(12-8x+x^2)]\\
&\quad +4(3\pi^2-14\pi x+14x^2)\big\}-0.01x^4(\pi-x)^4\big[\exp(-x)t^{3+\alpha}+1\big]\\
&\quad \big\{1-x^4(\pi-x)^4\big[\exp(-x)t^{3+\alpha}+1\big]\big\}  \quad\qquad x\in\Omega_x,~t\in(0,T],\\
u_0(x)~=&
0 \quad\qquad\qquad\qquad\qquad\qquad\qquad\qquad\qquad\qquad~ x\in\Omega_x,\\
\varphi(x,t)=&x^4(\pi-x)^4\big[\exp(-x)t^{3+\alpha}+1\big]\qquad\qquad\qquad~~~ x\in\partial\Omega_x,~t\in(0,T].
\end{aligned}
\end{equation}
It is clear that the nonlinear problem (\ref{eq:IVP-N}) has the following exact solution
\begin{equation}
u(x,t)=x^4(\pi-x)^4\big[\exp(-x)t^{3+\alpha}+1\big]\qquad x\in \Omega_x,~t\in[0,T].
\end{equation}

We test the accuracy of the FAOM-P$K$ algorithm for the nonlinear fractional diffusion equation with $\alpha=0.25$, 0.5, and 0.9.
To understand the accuracy of the FAOM-P$K$ scheme in time, we solve the problem with different time step sizes $h = 1/10, 1/20, 1/40, 1/80, 1/160, $
and a fixed spatial mesh size $\Delta x= \pi/5000$. For the reason of comparison, we also simulate the example by the direct method,
i.e., the $L1$ formula and $L1-2$ formula.
The computational errors and numerical convergence orders for the different methods with $\alpha=0.25$, 0.5, 
and 0.9 are given in Table \ref{Table:4-1}. As reported in the table, both the direct method
and the FAOM-P$K$ algorithm reach the ideal convergence orders, which are $2 -\alpha$ and $3 -\alpha$, respectively.
To check the convergence rate of our FAOM-P$K$ method in space, we simulate the case
with different spatial mesh sizes $\Delta x = \pi/80,\,\pi/40,\,/160,\,\pi/320,\,\pi/640$ and a fixed time step size $h =2^{-14}$.
The results are given in Table  \ref{Table:4-2}, which clearly shows that the FAOM-P$K$ has almost the same accuracy as
the corresponding direct method, but takes much less computational time.

\begin{table}[htb]\caption{The errors and convergence orders in time with fixed spatial mesh size
$\Delta x=\pi/5000$ for the proposed methods. 
$T=1$, ${\cal N}_{\tau}=2$. 
}\label{Table:4-1}
\begin{center}
\begin{tabular}{|c|cc|cc|cc|cc|}\hline
&\multicolumn{4}{c|}{$L1$ formula} &   \multicolumn{4}{|c|}{$L1-2$ formula}
\\
\hline
 \multirow{2}*{$h$}
  &   \multicolumn{2}{|c|}{\underline{Direct scheme}} &   \multicolumn{2}{|c|}{\underline{FAOM-P$4$}}
  &  \multicolumn{2}{|c|}{\underline{Direct scheme}}&   \multicolumn{2}{|c|}{\underline{FAOM-P$9$}} \\
   & $E(\cdot,\cdot)$  & $r_t$ & $E(\cdot,\cdot)$  & $r_t$& $E(\cdot,\cdot)$  & $r_t$& $E(\cdot,\cdot)$  & $r_t$\\ 
 \hline
      \multicolumn{9}{|c|}{$\alpha=0.9$} \\ 
 \hline
 1/10   &  3.72e-1	  & 1.25 & 3.72e-1  & 1.25&  6.76e-2  & 2.18& 6.76e-2  & 2.18\\ 
 1/20   & 1.56e-1 	  & 1.19 & 1.56e-1  & 1.19& 1.49e-2  & 2.16& 1.49e-2  & 2.16\\ 
 1/40   & 6.86e-2 	  & 1.15 & 6.87e-2  & 1.14& 3.33e-3  & 2.13& 3.33e-3 & 2.13\\ 
 1/80   & 3.10e-2         & 1.12 & 3.10e-2  & 1.12& 7.61e-4  & 2.11& 7.61e-4  & 2.11\\ 
 1/160 & 1.42e-2 	  & -      & 1.43e-2  & -     & 1.76e-4  & -    & 1.77e-4       & -\\ 
 \hline
       \multicolumn{9}{|c|}{$\alpha=0.5$} \\ 
 \hline
 1/10   & 1.19e-1	  & 1.68 & 1.19e-2  & 1.68& 2.06e-2  & 2.70& 2.06e-2  & 2.70\\ 
 1/20   & 3.71e-2 	  & 1.65 & 3.71e-2  & 1.64& 3.17e-3  & 2.70& 3.16e-3  & 2.69\\ 
 1/40   & 1.18e-2 	  & 1.61 & 1.19e-2  & 1.61& 4.91e-4  & 2.63& 4.91e-4  & 2.63\\ 
 1/80   & 3.87e-3         & 1.60 & 3.89e-3  & 1.57& 7.94e-5  & 2.42& 7.94e-5  & 2.41\\ 
 1/160 & 1.29e-3 	  & -       & 1.31e-3  & -     & 1.48e-5  &  -& 1.49e-5  & -     \\ 
 \hline
        \multicolumn{9}{|c|}{$\alpha=0.25$} \\ 
 \hline
 1/10   &  7.22e-2	  & 1.88 & 7.22e-2  & 1.88& 1.27e-2  & 2.91& 1.27e-2  & 2.91\\ 
 1/20   &  1.96e-2 	  & 1.89 & 1.96e-2  & 1.89& 1.70e-3  & 2.90& 1.70e-3  & 2.90\\ 
 1/40   &  5.30e-3 	  & 1.89 & 5.31e-3  & 1.88& 2.27e-4  & 2.82& 2.27e-4  & 2.82\\ 
 1/80   &  1.43e-3        & 1.87 & 1.44e-3  & 1.86 &  3.22e-5  & 2.28 & 3.22e-5  & 2.28\\ 
 1/160 &  3.91e-4 	  & -       & 3.97e-4  & -      & 6.64e-6  & - & 6.64e-6  & -      \\ 
 \hline
\end{tabular}
\end{center}
\end{table}

\begin{table}[htb]\caption{The errors and convergence orders in space with fixed time step size
$h=2^{-14}$ for the proposed methods. 
 ${\cal N}_{\tau}=2$, $T=1$, $\alpha=0.25$. Here CPU denotes the total compute time on the finest mesh.
}\label{Table:4-2}
\begin{center}
\begin{tabular}{|c|cc|cc|cc|cc|}\hline
&\multicolumn{4}{c|}{$L1$ formula} &   \multicolumn{4}{|c|}{$L1-2$ formula}
\\
\hline
 \multirow{2}*{$\Delta x$}
  &   \multicolumn{2}{|c|}{\underline{Direct scheme}} &   \multicolumn{2}{|c|}{\underline{FAOM-P$4$}}
  &  \multicolumn{2}{|c|}{\underline{Direct scheme}}&   \multicolumn{2}{|c|}{\underline{FAOM-P$9$}} \\
   & $E(\cdot,\cdot)$  & $r_s$ & $E(\cdot,\cdot)$  & $r_s$& $E(\cdot,\cdot)$  & $r_s$& $E(\cdot,\cdot)$  & $r_s$\\ 
 \hline
 $\pi/80$   &  1.12e-2	  & 2.00 & 1.12e-2  & 2.00&  1.12e-2  & 2.00& 1.12e-2  & 2.00\\ 
 $\pi/160$   & 2.81e-3 	  & 2.00 & 2.80e-3  & 2.00& 2.81e-3  & 2.00& 2.80e-3  & 2.00\\ 
 $\pi/320$   & 7.01e-4 	  & 2.00 & 7.02e-4  & 1.98& 7.01e-4  & 2.00& 7.01e-4 & 2.00\\ 
 $\pi/640$   & 1.75e-4         & - & 1.78e-4  & -& 1.75e-4  & -& 1.75e-4  & -\\ 
 \hline
CPU(s) & \multicolumn{2}{|c|}{807.43}     & \multicolumn{2}{|c|}{40.50}     &  \multicolumn{2}{|c|}{1451.29}   &  \multicolumn{2}{|c|}{70.85}\\ 
 \hline
  \end{tabular}
\end{center}
\end{table}

We next consider the initial value problem of the nonlinear fractional diffusion equation on the unbounded domain as follows 
\begin{equation}\label{eq:IVP-N-1}
\begin{aligned}
&^C_0{\cal D}^\alpha_t u(x,t)=u_{xx}(x,t)+f(u)\qquad &x\in \mathbb R, ~t>0,\\
&u(x,0)=u_0(x)\qquad & x\in\mathbb R.\\
\end{aligned}
\end{equation}
By setting $f(u)=-u(1-u)$, (\ref{eq:IVP-N-1}) is the time fractional Fisher equation
which is used in an infinite medium \cite{fisher}, the chemical kinetics \cite{malflict}, flame propagation \cite{frank},
and many other scientific problems \cite{merdan}. By setting $f(u)=-0.1u(1-u)(u-0.001)$, (\ref{eq:IVP-N-1})
is the time fractional Huxley equation, which is used to describe the transmission of nerve impulses \cite{Fitzhugh, Nagumo} with many
applications in biology and the population genetics in circuit theory \cite{Shih}.

If the initial data $u_0(x)$ is compactly supported on $\Omega_x=[a,b]$, 
we solve the fractional diffusion equation (\ref{eq:IVP-N-1}) on a bounded domain $[a,b]$
with absorbing bounder conditions (ABCs) \cite{lidf16}. The inner points is still discretized by the first equation of (\ref{eq:4.2}) and
the discretization of the points on the boundary
is given by the ABCs as follows
\begin{equation}\label{eq:4.10}
\begin{aligned}
&(\tilde \delta _x+3s_0^{\frac \alpha 2} )^C_0{\cal D}^{F,\alpha}_t u(x_{N-1},t_{j+1})+
(3s_0^{\alpha }\tilde \delta _x+s_0^{\frac {3\alpha} 2})u_{N-1}^{j+1}=(\tilde \delta _x+3s_0^{\frac \alpha 2})f(\tilde u_{N-1}^{j+1}),\\
&(\tilde \delta _x-3s_0^{\frac \alpha 2} )^C_0{\cal D}^{F,\alpha}_t u(x_{1},t_{j+1})+
(3s_0^{\alpha }\tilde \delta _x-s_0^{\frac {3\alpha} 2})u_{1}^{j+1}=(\tilde \delta _x-3s_0^{\frac \alpha 2})f(\tilde u_{1}^{j+1}),
\end{aligned}
\end{equation}
where $\tilde{\delta}_x u_i^{j+1}=\frac{u_{i+1}^{j+1}-u_{i-1}^{j+1}}{2\Delta x}$ and $s_0=3$ according to  \cite{lidf16}.

\noindent\textbf{Example 4.2} We consider the time fractional Fisher equation with $f(u)=-u(1-u)$ in (\ref{eq:IVP-N-1}) and 
initial condition 
$$
u(x,0)=\sqrt{\frac {10} \pi}\exp {(-10 x^2)}.
$$
The computational domain is set as $[-6,6]$.
Since it is difficult to obtain the exact solution of the time fractional equation on the unbounded domain, here and below we take the solution on
a very fine mesh as the reference solution. Table \ref{Table:4-3} presents the numerical results
for $\alpha=0.25$ and $0.75$, which shows that the FAOM-P$K$ algorithm has the same convergence order
in time as the corresponding direct scheme. The convergence orders in time of the direct scheme
and the FAOM-P$K$ algorithm are higher than $1$, but lower than the ideal convergence order.
To understand it, we plot the numerical solution in Fig. \ref{fig:1-6}, which shows that $u'(t)$ has singularity at $t=0$.
How to obtain a high accurate method for the solution with singularity is still an unsolved problem.  
Due to the approximation of $f(u)$, the convergence rate in time of our scheme is high than the scheme reported 
in \cite{jiang, lidf16}. Table  \ref{Table:4-4} indicates that the FAOM-P$K$ algorithm has the second-order of accuracy in space and  
takes less computational time than the direct algorithm.

\begin{table}[htb]\caption{The errors and convergence orders for the Fisher equation in time with fixed spatial mesh size
$\Delta x=3\times2^{-10}$ for the proposed methods. 
 ${\cal N}_{\tau}=2$, $T=1$. 
}\label{Table:4-3}
\begin{center}
\begin{tabular}{|c|cc|cc|cc|cc|}\hline
&\multicolumn{4}{c|}{$L1$ formula} &   \multicolumn{4}{|c|}{$L1-2$ formula}
\\
\hline
 \multirow{2}*{$h$}
  &   \multicolumn{2}{|c|}{\underline{Direct scheme}} &   \multicolumn{2}{|c|}{\underline{FAOM-P$4$}}
  &  \multicolumn{2}{|c|}{\underline{Direct scheme}}&   \multicolumn{2}{|c|}{\underline{FAOM-P$9$}} \\
   & $\|\mathbf{e}^{N_T}\|_\infty$  & $r_t$ & $\|\mathbf{e}^{N_T}\|_\infty$  & $r_t$& $\|\mathbf{e}^{N_T}\|_\infty$  & $r_t$& $\|\mathbf{e}^{N_T}\|_\infty$  & $r_t$\\ 
 \hline
      \multicolumn{9}{|c|}{$\alpha=0.25$} \\ 
 \hline
$2^{-8}$   &  1.92e-4	  & 1.15 & 1.93e-4  & 1.14&  1.81e-4  & 1.16& 1.88e-4  & 1.15\\ 
$2^{-9}$  & 8.63e-5 	  & 1.27 & 8.74e-5  & 1.24& 8.12e-5  & 1.28& 8.49e-5  & 1.26\\ 
$2^{-10}$   & 3.57e-5  & 1.63 & 3.70e-5  & 1.52& 3.35e-5  & 1.63& 3.54e-5 & 1.59\\ 
$2^{-11}$   & 1.16e-5  & - 	     & 1.29e-5  & -      & 1.08e-5  & -& 1.17e-5  & -\\ 
 \hline     \multicolumn{9}{|c|}{$\alpha=0.75$} \\ 
 \hline
$2^{-8}$   &  3.19e-4	  & 1.21 & 3.18e-4  & 1.21&  2.27e-4  & 1.36& 2.36e-4  & 1.34\\ 
$2^{-9}$     & 1.38e-4  & 1.29 & 1.37e-4  & 1.30& 8.85e-5  & 1.44& 9.34e-5  & 1.41\\ 
$2^{-10}$   & 5.65e-5  & 1.62 & 5.58e-5  & 1.66& 3.26e-5  & 1.75& 3.52e-5 & 1.68\\ 
$2^{-11}$   & 1.83e-5  & -       & 1.76e-5  &       -& 9.67e-6  & -      & 1.10e-5  & -\\ 
 \hline
  \end{tabular}
\end{center}
\end{table}

\begin{figure}[htb!]
\begin{center}
{\includegraphics[width=0.47\textwidth]{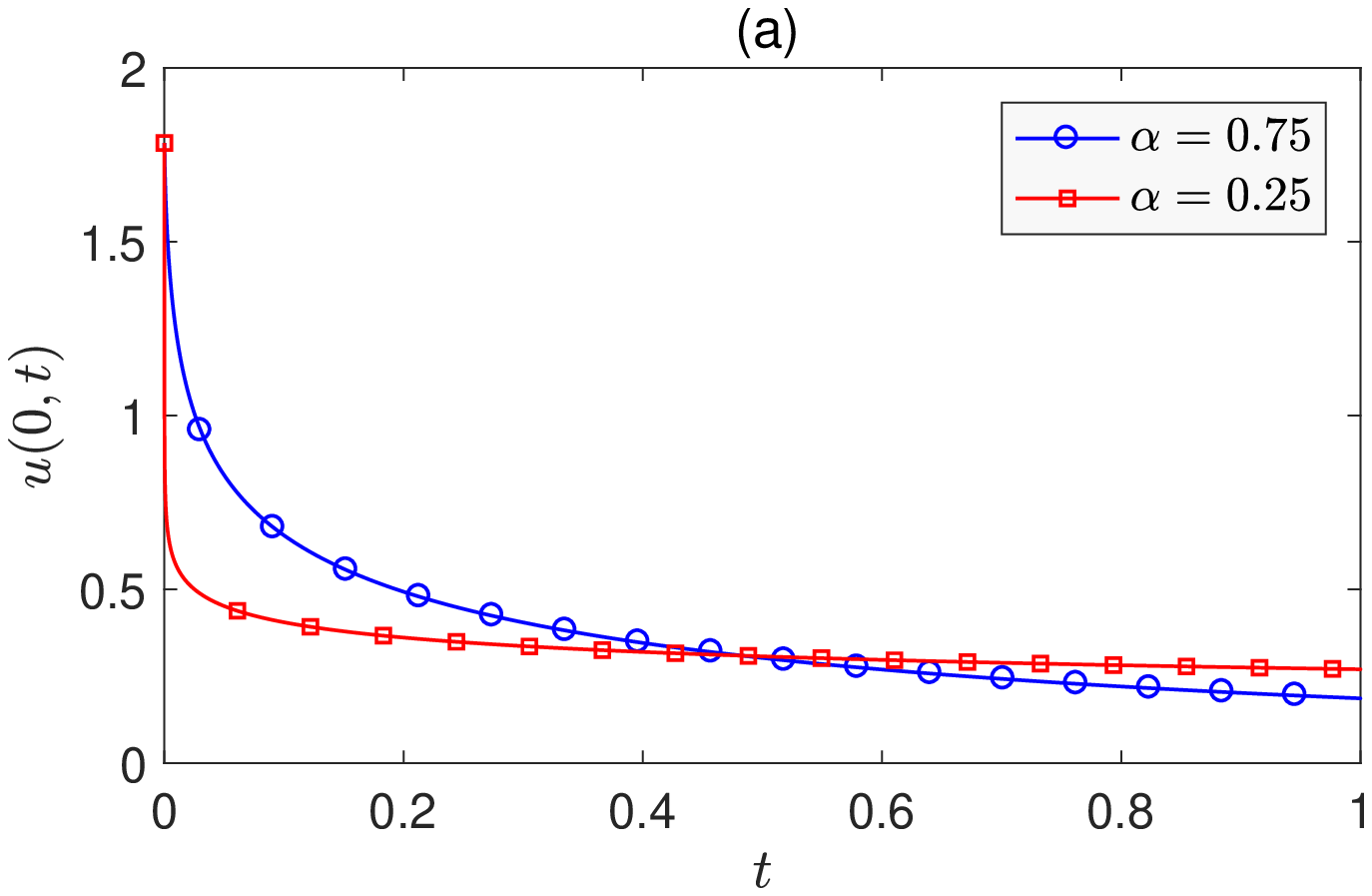}}
{\includegraphics[width=0.47\textwidth]{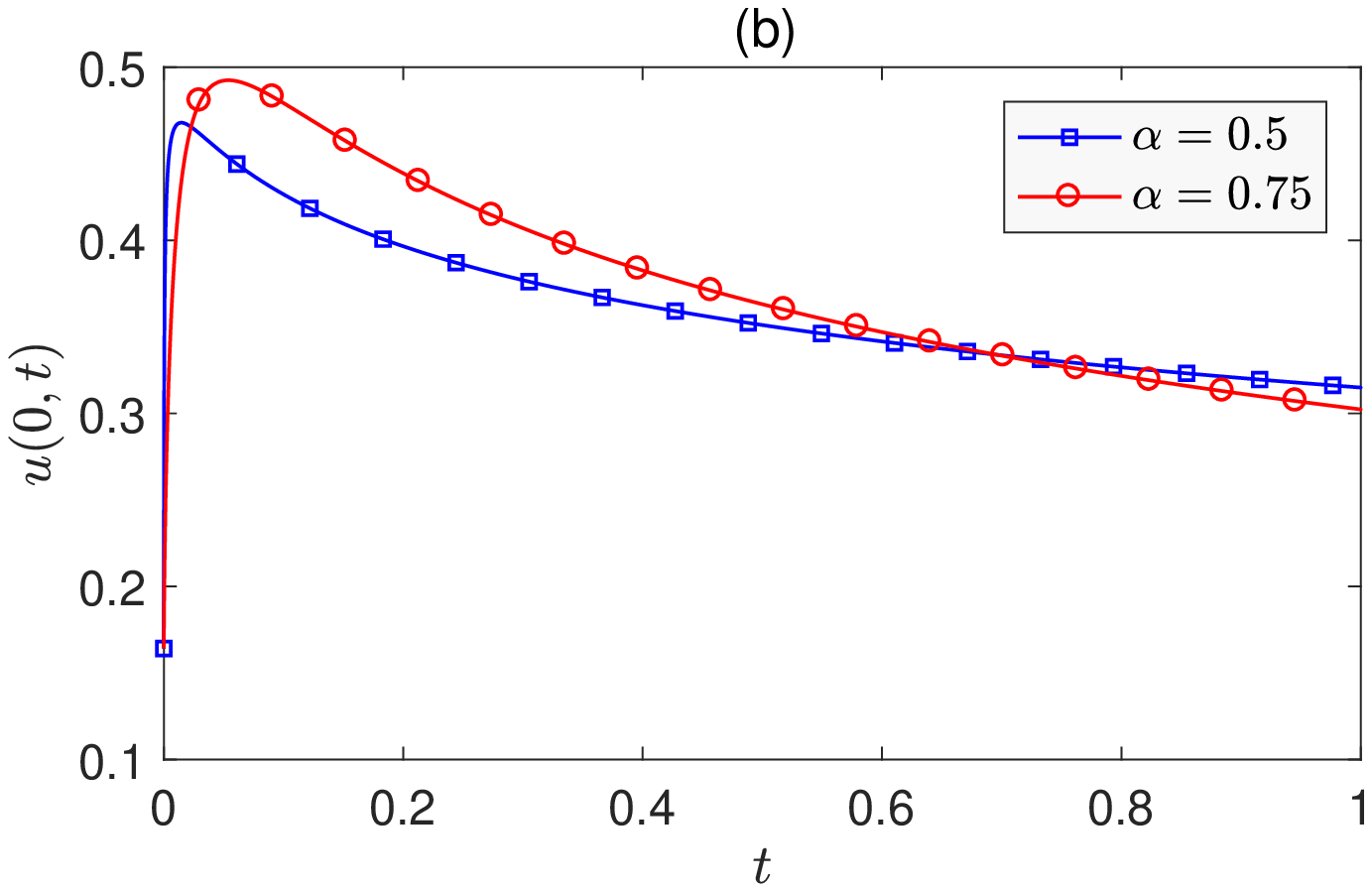}}
\end{center}
\caption{
Numerical solutions at $x=0$. (a) The time fractional Fisher equation with $\alpha=0.25,\,0.75$. (b) The time fractional
Huxley equation with $\alpha=0.5,\,0.75$.
  }\label{fig:1-6}
\end{figure}

\begin{table}[htb]\caption{The errors and convergence orders for the Fisher equation in space with fixed time step size
$h=2^{-14}$ for the proposed methods. 
 ${\cal N}_{\tau}=2$, $T=1$. Here CPU denotes the total compute time on the finest mesh.
}\label{Table:4-4}
\begin{center}
\begin{tabular}{|c|cc|cc|cc|cc|}
\hline
 \multirow{2}*{$\Delta x$}
  &   \multicolumn{2}{|c|}{\underline{Direct scheme}} &   \multicolumn{2}{|c|}{\underline{FAOM-P$4$}}
  &  \multicolumn{2}{|c|}{\underline{Direct scheme}}&   \multicolumn{2}{|c|}{\underline{FAOM-P$4$}} \\
   & $\|\mathbf{e}^{N_T}\|_\infty$  & $r_s$ & $\|\mathbf{e}^{N_T}\|_\infty$  & $r_s$& $\|\mathbf{e}^{N_T}\|_\infty$  & $r_s$& $\|\mathbf{e}^{N_T}\|_\infty$  & $r_s$\\ 
 \hline
   &   \multicolumn{4}{|c|}{$\alpha=0.25$} & \multicolumn{4}{|c|}{$\alpha=0.75$}\\ 
 \hline
$3/2^{5}$  &  6.92e-4	  & 2.03 & 6.94e-4  & 2.02&  3.53e-4  & 2.03& 3.53e-4  & 2.04\\ 
$3/2^{6}$   & 1.69e-4 	  & 2.07 & 1.71e-4  & 2.04& 8.66e-5  & 2.07& 8.58e-5  & 2.12\\ 
$3/2^{7}$   & 4.03e-5 	  & 2.32 & 4.17e-5  & 2.14& 2.06e-5  & 2.32& 1.98e-5 & 2.60\\ 
$3/2^{8}$   & 8.05e-6         & - 	    & 9.48e-6  & -	& 4.12e-6  & -& 3.34e-6  & -\\ 
 \hline
CPU(s) & \multicolumn{2}{|c|}{938.50}     & \multicolumn{2}{|c|}{33.46}     &  \multicolumn{2}{|c|}{905.82}   &  \multicolumn{2}{|c|}{29.15}\\ 
 \hline
  \end{tabular}
\end{center}
\end{table}

\noindent\textbf{Example 4.3}. We consider the time fractional Huxley equation with $f(u)=-0.1u(1-u)(u-0.001)$ in (\ref{eq:IVP-N-1}) and 
initial condition 
$$
u(x,0)=\exp {\big(-10 (x-0.5)^2)}+\exp {\big(-10 (x+0.5)^2)}.
$$
The computational domain is set as $[-8,8]$.
Table \ref{Table:4-5} presents the numerical results
for $\alpha=0.5$, which shows that the FAOM-P$K$ algorithm has the same convergence order
in time as the corresponding direct scheme. Similar to the above example, the convergence orders in time of the direct scheme
and the FAOM-P$K$ algorithm are higher than $1$, but lower that the ideal convergence order.
Due to the approximation of $f(u)$, the convergence order in time of our scheme is higher than the scheme reported 
in \cite{jiang, lidf16}. Table  \ref{Table:4-6} indicates that the FAOM-P$K$ algorithm has the second-order of accuracy in space and  
takes less computational time than the direct algorithm.

\begin{table}[htb]\caption{The errors and convergence orders for the Huxley equation in time with fixed spatial mesh size
$\Delta x=2^{-6}$ for the proposed methods. 
 ${\cal N}_{\tau}=2$, $T=1$, $\alpha=0.5$. 
}\label{Table:4-5}
\begin{center}
\begin{tabular}{|c|cc|cc|cc|cc|}\hline
&\multicolumn{4}{c|}{$L1$ formula} &   \multicolumn{4}{|c|}{$L1-2$ formula}
\\
\hline
 \multirow{2}*{$h$}
  &   \multicolumn{2}{|c|}{\underline{Direct scheme}} &   \multicolumn{2}{|c|}{\underline{FAOM-P$4$}}
  &  \multicolumn{2}{|c|}{\underline{Direct scheme}}&   \multicolumn{2}{|c|}{\underline{FAOM-P$9$}} \\
   & $\|\mathbf{e}^{N_T}\|_\infty$  & $r_t$ & $\|\mathbf{e}^{N_T}\|_\infty$  & $r_t$& $\|\mathbf{e}^{N_T}\|_\infty$  & $r_t$& $\|\mathbf{e}^{N_T}\|_\infty$  & $r_t$\\ 
 \hline
$2^{-5}$   &  8.94e-4	  & 1.13 & 8.94e-4  & 1.12&  5.49e-4  & 1.10& 6.29e-4  & 1.10\\ 
$2^{-6}$  & 4.10e-4 	  & 1.24 & 4.10e-4  & 1.24& 2.56e-4  & 1.23& 2.94e-4  & 1.21\\ 
$2^{-7}$   & 1.73e-4   & 1.60 & 1.74e-4   & 1.58& 1.09e-4  & 1.60& 1.28e-4 & 1.50\\ 
$2^{-8}$   & 5.72e-5   & -       & 5.81e-5   & -      & 3.59e-5  & -& 4.51e-5  & -\\ 
 \hline
  \end{tabular}
\end{center}
\end{table}

\begin{table}[htb]\caption{The errors and convergence orders for the Huxley equation in space with fixed time step size
$h=2^{-14}$ for the proposed methods. 
 ${\cal N}_{\tau}=2$, $T=1$. Here CPU denotes the total compute time on the finest mesh.
}\label{Table:4-6}
\begin{center}
\begin{tabular}{|c|cc|cc|cc|cc|}\hline
 \multirow{2}*{$\Delta x$}
  &   \multicolumn{2}{|c|}{\underline{Direct scheme}} &   \multicolumn{2}{|c|}{\underline{FAOM-P$4$}}
  &  \multicolumn{2}{|c|}{\underline{Direct scheme}}&   \multicolumn{2}{|c|}{\underline{FAOM-P$4$}} \\
   & $\|\mathbf{e}^{N_T}\|_\infty$  & $r_s$ & $\|\mathbf{e}^{N_T}\|_\infty$  & $r_s$& $\|\mathbf{e}^{N_T}\|_\infty$  & $r_s$& $\|\mathbf{e}^{N_T}\|_\infty$  & $r_s$\\ 
   \hline
   &\multicolumn{4}{c|}{$\alpha=0.5$} &   \multicolumn{4}{|c|}{$\alpha=0.75$}
\\
\hline
$2^{-2}$   &  2.14e-3	  & 2.14 & 2.14e-3  & 2.13& 1.41e-3  & 2.12& 1.41e-3  & 2.12\\ 
$2^{-3}$   & 4.88e-4 	  & 2.10 & 4.89e-4  & 2.08& 3.24e-4  & 2.09& 3.24e-4  & 2.10\\ 
$2^{-4}$  & 1.14e-4 	  & 2.33 & 1.15e-4  & 2.27& 7.61e-5  & 2.33& 7.57e-5 & 2.36\\ 
$2^{-5}$  & 2.27e-5         & - & 2.40e-5  & -	& 1.52e-5  & -& 1.47e-5  & -\\ 
 \hline
CPU(s) & \multicolumn{2}{|c|}{743.33}     & \multicolumn{2}{|c|}{25.81}     &  \multicolumn{2}{|c|}{773.61}   &  \multicolumn{2}{|c|}{25.73}\\ 
 \hline
  \end{tabular}
\end{center}
\end{table}

\section{Conclusions}
In this paper, we present a high order fast algorithm with almost optimum memory for the Caputo fractional derivative. 
The fast algorithm is based on a nonuniform split of the interval $[0,t_n]$ and a polynomial approximation of the kernel
function $(1-\tau)^{-\alpha}$, in which the storage requirement and computational cost both are reduced from $O(n)$ to $O(\log n)$. 
We prove that the fast algorithm has the same convergence rate as that of the corresponding direct method, even a high order scheme is compared.
The fast algorithm is applied to solve the linear and nonlinear fractional diffusion equations.
Numerical results on linear and nonlinear fractional diffusion equations show that our fast scheme has the same order of
convergence as the corresponding direct methods, but takes much less computational time.


\begin{thebibliography}{99}
\bibitem{Baffet}
{\sc D. Baffet and J. S. Hesthaven}, {\em A kernel compression scheme for fractional differential equations}, 
SIAM J. Numer. Anal., 55(2) (2017), pp. 496--520.

\bibitem{cao}
{\sc J. Cao and C. Xu}, {\em A high order scheme for the numerical solution of the fractional ordinary
differential equations}, J. Comput. Phys., 238 (2013), pp. 154--168.



\bibitem{chen}
{\sc C. Chen, F. Liu, I. Turner, and V. Anh}, {\em A Fourier method for the fractional diffusion equation describing sub-diffusion}, J. Comput. Phys., 227 (2007), pp. 886--897.

\bibitem{cui}
{\sc M. Cui}, {\em Compact finite difference method for the fractional diffusion equation}, J. Comput. Phys., 228 (2009), pp. 7792--7804.


\bibitem{fisher}
{\sc R. A. Fisher}, {\em The wave of advance of advantageous genes}, Ann. Eugene, 7 (1937), pp. 335--369.
\bibitem{Fitzhugh}
{\sc R. Fitzhugh}, {\em Impulse and physiological states in models of nerve membrane}, Biophys. J, 1
(1961), pp. 445--466.

\bibitem{frank}
{\sc D. A. Frank}, {\em Diffusion and heat exchange in chemical kinetics}, Princeton University Press,
Princeton, NJ, USA.

\bibitem{gao14}
{\sc G. H. Gao, Z. Z. Sun and H. W. Zhang}, {\em A new fractional numerical differentiation formula to approximate
the Caputo fractional derivative and its applications}, J. Comput. Phys., 259 (2014), pp. 33--50.

\bibitem{gao12}
{\sc G. H. Gao, Z. Z. Sun, and Y. N. Zhang}, {\em A finite difference scheme for fractional sub-diffusion
equations on an unbounded domain using artificial boundary conditions}, J. Comput. Phys.,
231 (2012), pp. 2865--2879.

\bibitem{gao13}
{\sc G. H. Gao and Z. Z. Sun}, {\em The finite difference approximation for a class of fractional
sub-diffusion equations on a space unbounded domain}, J. Comput. Phys.,
236 (2013), pp. 443--460.

%






\bibitem{jiang}
{\sc S. D. Jiang, J. W. Zhang, Q. Zhang, and Z. M. Zhang}, {\em Fast evaluation of the Caputo fractional derivative and its applications to fractional diffusion equations}, 
Commun. Comput. Phys., 21(3) (2017), pp. 650--678.



\bibitem{kilbas}
{\sc A. Kilbas, H. Srivastava, and J. Trujillo}, {\em Theory and Applications of Fractional Differential
Equations}, Elesvier Science and Technology, Boston, 2006.


\bibitem{langlands}
{\sc T. Langlands and B. Henry}, {\em The accuracy and stability of an implicit solution method for
the fractional diffusion equation}, J. Comput. Phys., 205 (2005), pp. 719--736.


\bibitem{langlands10}
{\sc T. Langlands and B. Henry}, {\em Fractional chemotaxis diffusion equations}, Phys. Rev. E, 81
(2010), pp. 051102.



%
\bibitem{lidf16}
{\sc D. F. Li and J. W. Zhang}, {\em Efficient implementation to numerically solve the nonlinear time fractional parabolic problems 
on unbounded spatial domain}, J. Comput. Phys., 322 (2016), pp. 415--428.

\bibitem{lijr2010}
{\sc J. R. Li}, {\em A fast time stepping method for evaluating fractional integrals}. SIAM J. Sci. Comput., 31(6) (2010), pp. 4696--4714.

\bibitem{li09}
{\sc X. Li and C. Xu}, {\em A Space-time spectral method for the time fractional diffusion equation}, SIAM J. Numer. Anal., 47 (2009), pp. 2108--2131.

\bibitem{lin11}
{\sc Y. Lin, X. Li, and C. Xu}, {\em Finite difference/spectral approximations for the fractional cable
equation}, Math. Comp. 80 (2011), pp. 1369--1396.

\bibitem{lin07}
{\sc Y. Lin and C. Xu}, {\em Finite difference/spectral approximations for the time-fractional diffusion
equation}, J. Comput. Phys., 225 (2007) pp. 1533--1552.

\bibitem{lopez2008}	
{\sc M. L\'opez-Fern\'andez, C. Lubich, and A. Sch\"adle}, {\em Adaptive, fast, and oblivious convolution in evolution equations with memory}, SIAM J. Sci. Comput., 30(2) (2008), pp. 1015--1037.

\bibitem{lubich2002}
{\sc C. Lubich and A. Sch\"{a}dle}, {\em Fast convolution for nonreflecting boundary conditions}, SIAM J. Sci. Comput.,  24(1) (2002), pp. 161--182.





\bibitem{malflict}
{\sc W. Malflict}, {\em Solitary wave solutions of nonlinear wave equations}, Am. J. Phys., 60 (1992),
pp. 650--654.

\bibitem{mclean2012}
{\sc W. McLean}, {\em Fast summation by interval clustering for an evolution equation with memory}, SIAM J. Sci. Comput., 34(6) (2012), pp. A3039--A3056.


\bibitem{merdan}
{\sc M. Merdan}, {\em Solutions of time-fractional reaction-diffusion equation with modified 
Riemann--Liouville derivative}, Int. J. Phys. Sci., 7 (2012), pp. 2317--2326.

\bibitem{metzler}
{\sc R. Metzler and J. Klafter}, {\em The random walks guide to anomalous diffusion: a fractional
dynamics approach}, Phys. Rep., 339 (2000), pp. 1--77.

\bibitem{Nagumo}
{\sc J. S. Nagumo, S. Arimoto, and S. Yoshizawa}, {\em An active pulse transmission line simulating
nerve axon}, Proc. IRE, 50 (1962), pp. 2061--2070.

\bibitem{oldham}
{\sc K. B. Oldham and J. Spanier}, {\em The Fractional Calculus}, Academic Press, New York, 1974.








\bibitem{podlubny99}
{\sc I. Podlubny}, {\em Fractional Differential Equations}, Academic Press, NewYork, 1999.





\bibitem{ren}
{\sc J. Ren, Z. Sun, and W. Dai}, {\em New approximations for solving the Caputo-type fractional partial differential equations}, Appl. Math. Model., 40 (2016), pp. 2625--2636.


\bibitem{schadle2006}
{\sc A. Sch\"adle, M. L\'{o}pez-Fern\'{a}ndez, and C. Lubich}, {\em Fast and oblivious convolution quadrature}, SIAM J. Sci. Comput., 28(2) (2006), pp. 421--438.

\bibitem{Shih}
{\sc M. Shih, E. Momoniat, and F. M. Mahomed}, {\em Approximate conditional symmetries and approximate
solutions of the perturbed Fitzhugh--Nagumo equation}, J. Math. Phys., 46 (2005),
pp. 023503.
\bibitem{sun06}
{\sc Z. Sun and X. Wu}, {\em A fully discrete difference scheme for a diffusion-wave system}, Appl.
Numer. Math., 56 (2006), pp. 193--209.













%


\bibitem{yang}
{\sc Q. Yang, I. Turner, F. Liu, and M. Ilis}, {\em Novel numerical methods for solving the time-space fractional diffusion equation in 2D}, 
SIAM J. Sci. Comput., 33 (2011), pp. 1159--1180.

\bibitem{zeng}
{\sc F. Zeng, C. Li, F. Liu, and I. Turner}, {\em Numerical algorithms for time-fractional subdiffusion equation with second-order accuracy}, 
SIAM J. Sci. Comput.,  37 (2015), pp. A55--A78.

\bibitem{zeng2017}
{\sc F. Zeng, I. Turner, and K. Burrage}, {\em A stable fast time-stepping method for fractional integral and derivative operators}, arXiv:1703.05480, 2017.

\bibitem{zhang}
{\sc Y. N. Zhang, Z. Z. Sun, and H. W. Wu}, {\em Error estimates of Crank--Nicolson-type difference scheme for the subdiffusion equation}, SIAM J. Numer. Anal., 49 (2011), pp. 2302--2322.







\end{thebibliography}
\end{document}